\documentclass[12pt]{amsart}
\usepackage{epsfig,amssymb,amsfonts,amsthm}
\usepackage{graphicx}
\usepackage{bm}
\newtheorem{thm}{Theorem}[section]

\newtheorem{claim}[thm]{Claim}

\theoremstyle{definition}
\newtheorem{de}[thm]{Definition}
\theoremstyle{remark}
\newtheorem{rem}[thm]{Remark}
\numberwithin{equation}{section}

\def\M{\mathcal M}
\def\lam_2{\lambda_2}

\begin{document}
\title[Convergence to equilibrium and Talagrand-type inequalities]{Convergence to global equilibrium for Fokker-Planck
equations on a graph and Talagrand-type inequalities}

\author[R. Che]{Rui Che}
\address{R. Che: Wu Wen-Tsun Key Laboratory of Mathematics, USTC, Chinese Academy of Sciences, Hefei Anhui 230026, PRC}
\email{boboke-1126@163.com}
\author[W. Huang]{Wen Huang}
\address{W. Huang: Wu Wen-Tsun Key Laboratory of Mathematics, USTC, Chinese Academy of Sciences, Hefei Anhui 230026, PRC}
\email{wenh@mail.ustc.edu.cn}

\author[Y. Li]{Yao Li}
\address{Y. Li: Courant Institute of Mathematical Sciences, New York University, New York, NY 10012, USA} \email{yaoli@cims.nyu.edu}
\author[P. Tetali]{Prasad Tetali}
\address{P. Tetali: School of Mathematics, Georgia Institute
of Technology, Atlanta, GA 30332, USA} \email{tetali@math.gatech.edu}

\subjclass[2000]{Primary: 37H10,60J27,60J60.}

\keywords{Fokker-Planck equation,  Gibbs density, graph, Talagrand inequality}

\thanks{The second author is supported by NSFC (No 11225105).  The work of the last author is supported in part by NSF DMS 1101447.}

\begin{abstract}
In recent work, Chow, Huang, Li and Zhou \cite{CHLZ} introduced the
study of  Fokker-Planck equations for a free energy function defined
on a finite graph.  When $N\ge 2$ is the number of vertices of the
graph, they show that  the corresponding Fokker-Planck equation is a
system of $N$ nonlinear ordinary differential equations defined on a
Riemannian manifold of probability distributions.  The different
choices for inner products on the space of probability distributions
result in different Fokker-Planck equations for the same process.
Each of these Fokker-Planck equations has a unique global
equilibrium, which is a Gibbs distribution. In this paper we study
the {\em speed of convergence} towards global equilibrium for the
solution of  these Fokker-Planck equations on a graph, and prove
that the convergence is indeed exponential. The rate as measured by
the decay of the $L_2$ norm can be bound in terms of  the spectral
gap of the Laplacian of the graph, and as measured by the decay of
(relative) entropy be bound using the modified logarithmic Sobolev
constant of the graph.

With the convergence result, we also prove two Talagrand-type
inequalities relating relative entropy and Wasserstein metric, based
on two different metrics introduced in \cite{CHLZ}. The first one is
a local inequality, while the second is a global inequality with
respect to the ``lower bound metric" from \cite{CHLZ}.
\end{abstract}

\maketitle

\section{Introduction}

As the stochastic differential equation becomes one of the
primary and  highly effective tools in many practical problems arising in diverse fields
such as finance, physics, chemistry and
biology \cite{Gar,Ris,Sch},  there are considerable efforts in understanding
the properties of the classical Fokker-Planck equation that describes
the time evolution of the probability distribution of a stochastic
process. At the same time, the free energy functional,  which  is defined on the
space of probability distributions, as a linear combination of  terms involving a
potential and an entropy, has widely been used in various subjects; it typically means different things in different
contexts. For example,  the notion of ``free energy" in thermodynamics is related to
the maximal amount of work that can be extracted from a system. The
concept of free energy is also used in other fields, such as
statistical mechanics, probability (particularly in the context of Markov Random Fields),
biology, chemistry, and  image processing; see e.g., \cite{HT,SW,WHY}.

Since the seminal work of Jordan, Kinderlehrer and Otto \cite{JKO,
O}, it is well known that a Fokker-Planck equation is the gradient
flow of the free energy functional on a Riemannian manifold that is
defined by a space of probability distributions with a 2-Wasserstein
metric on it. This discovery has been the starting point for many
developments  relating the free energy, Fokker-Planck equation, an abstract
notion of a Ricci curvature and optimal transport theory in  continuous
spaces. We refer to the  monographs \cite{AGS, V, V2} for an
overview. Recently, a synthetic theory of Ricci curvature in length
spaces has been developed by Lott-Sturm-Villani \cite{LV, SK, SK2},
which reveals the fundamental relationship between entropy and Ricci
curvature. Despite the remarkable developments on this subject on a
continuous space, much less is known when the underlying space is
discrete, as in an (undirected) graph.


In recent work,  Chow, Huang, Li and Zhou  \cite{CHLZ} considered
Fokker-planck equations for a free energy function (or a certain Markov process)
defined on a finite graph. For a graph on $N\ge 2$ vertices,
they showed that the corresponding Fokker-Planck equation is a system of $N$ nonlinear
ordinary differential equations, defined on a Riemannian manifold of
probability distributions.  In fact, they point out that one could make different choices for inner products on the space of probability distributions resulting, in turn, in different Fokker-Planck equations for the same process. Furthermore, each of these systems of ordinary differential
equations has a unique global equilibrium -- a Gibbs distribution --  and is
a gradient flow for the free energy functional defined on a
Riemannian manifold whose  metric is closely related to certain classical
Wasserstein metrics.

We recall here, more formally, the  approach of Chow et al \cite{CHLZ}.
Consider a  graph $G = (V,E)$, where $V=\{a_{1},\cdots,a_{N}\}$ is
the set of vertices $|V|\ge 2$, and $E$ denotes the set of (undirected) edges.
For simplicity, assume that the graph is connected and is simple -- with no self-loops or multiple edges. Let $N(i) := \{j\in \{1,2,\cdots,N\}|\{a_{i},a_{j}\}\in E \}$ denote
the neighborhood of a vertex $a_{i}\in V$.

Let ${\bm \Psi}=(\Psi_{i})_{i=1}^{N}$ be a given {\em potential}
function on $V$, where $\Psi_i$ is the potential on vertex $a_i$.
Further denote
$$\mathcal{M}=\{ {\bm \rho} =
(\rho_{i})_{i=1}^{N}\in \mathbb{R}^N|\sum_{i=1}^N\rho_i=1\text{ and
}\rho_i>0 \text{ for }i=1,2,\cdots,N\},
$$ as the space of all positive probability distributions on $V$.

Then the free energy functional has the following expression: for each ${\bm \rho}\in \M$, let
\begin{equation}\label{eq-1-1}
   F({\bm \rho}) := F_{{\bm \Psi},\beta}({\bm \rho})= \sum_{i=1}^N \Psi_{i}\rho_{i}+\beta\sum_{i=1}^N \rho_{i}\log\rho_{i},
\end{equation}
where $\beta> 0$ is the strength of ``white noise'' or the temperature. The free energy functional has a global minimizer, called a Gibbs density, and is
given by
\begin{equation}\label{eq-1-3}
   {\rho}^{*}_{i} = \frac{1}{K}e^{-\Psi_{i}/\beta} , \qquad \text{ where }K = \sum_{i=1}^N e^{-\Psi_{i}/\beta}.
\end{equation}

From a free energy viewpoint, Chow et al \cite{CHLZ} endowed  the  space $\mathcal{M}$ with a
Riemannian metric $d_{\Psi}$, which depended on the potential $\Psi$
as well as the structure of the graph.
Then by considering the gradient flow of the free energy
\eqref{eq-1-1} on such a Riemannian manifold $(\mathcal{M},d_{\Psi})$,
they obtained a Fokker-Planck equation on $\mathcal{M}$:
\begin{align}\label{dfpe}
\begin{aligned}
 \frac{d \rho_{i}} {dt}=&\sum_{j\in N(i),\Psi_{j}>\Psi_{i}}
 ((\Psi_{j}+\beta\log\rho_{j})-(\Psi_{i}+\beta\log\rho_{i}))\rho_{j}\\
 &+\sum_{j\in N(i),\Psi_{j}<\Psi_{i}} ((\Psi_{j}+\beta\log\rho_{j})-(\Psi_{i}+\beta\log\rho_{i}))\rho_{i}  \\
&+\sum_{j\in N(i), \Psi_{j}=\Psi_{i}}\beta(\rho_{j}-\rho_{i})
\end{aligned}
\end{align}
for $i=1,2\cdots,N$(see Theorem 2 in \cite{CHLZ}).

From a stochastic process viewpoint, the work of Chow et al may be seen as a new interpretation of white noise perturbations to a Markov process on $V$.  By considering the time evolution
equation of its probability density function, they obtained another
Fokker-Planck equation on $\mathcal{M}$:
\begin{align}\label{dfpe2}
\begin{aligned}
 \frac{d \rho_{i}} {dt}=&\sum_{j\in N(i),\bar{\Psi}_{j}>\bar{\Psi}_{i}} ((\Psi_{j}+\beta\log\rho_{j})-(\Psi_{i}+\beta\log\rho_{i}))\rho_{j}\\
   &+\sum_{j\in N(i),\bar{\Psi}_{j}<\bar{\Psi}_{i}}
 ((\Psi_{j}+\beta\log\rho_{j})-(\Psi_{i}+\beta\log\rho_{i}))\rho_{i} \\
= &\sum_{j\in N(i)} ((\Psi_{j}+\beta\log\rho_{j})-(\Psi_{i}+\beta\log\rho_{i}))^+\rho_{j}\\
   &- \sum_{j\in N(i)} ((\Psi_{j}+\beta\log\rho_{j})-(\Psi_{i}+\beta\log\rho_{i}))^-\rho_{i}
\end{aligned}
\end{align}
for $i=1,2,\cdots,N$, where $\bar{\Psi}_i=\Psi_i+\beta\log \rho_i$
for $i=1,2,\cdots,N$ (see Theorem 3 in \cite{CHLZ}). For
convenience, we call equations (\ref{dfpe}) and (\ref{dfpe2})
Fokker-Planck equation I \eqref{dfpe} and II \eqref{dfpe2}
respectively. Both  \eqref{dfpe} and \eqref{dfpe2} share similar
properties for $\beta>0$ (see Theorem 2 and Theorem 3 in
\cite{CHLZ}):

\begin{enumerate}
\item Both equations are gradient flows of the same free
  energy on the same probability space $\mathcal{M}$, but with different metrics.

\item The Gibbs distribution ${\bm \rho}^{*}=(\rho_i^*)_{i=1}^N$,
given by \eqref{eq-1-3}, is the unique stationary distribution of
both  equations in $\mathcal{M}$. Furthermore, the free energy $F$
attains its global minimum at the Gibbs distribution.

\item  For both equations, given any initial condition ${\bm \rho}^0\in \mathcal{M}$,
there exists a unique solution
$${\bm \rho}(t):[0,\infty)\rightarrow \mathcal{M}$$
with the initial value ${\bm \rho}^0\in \mathcal{M}$, and ${\bm
\rho}(t)$ satisfying the properties:
    \begin{enumerate}
    \item the free energy $F({\bm \rho}(t))$ decreases as  time $t$ increases, and

     \item ${\bm \rho}(t)\rightarrow {\bm \rho}^*$ under the Euclidean metric
       of $\mathbb{R}^N$, as $t\rightarrow +\infty$.
    \end{enumerate}
\end{enumerate}

\medskip

There are differences between equations \eqref{dfpe} and
\eqref{dfpe2}. Fokker-Planck equation I \eqref{dfpe} is obtained
from the gradient flow of the free energy $F$ on the Riemannian
metric space $(\mathcal{M},d_{\bm \Psi})$. However, its connection
to a Markov process on the graph is not clear. On the other hand,
Fokker-Planck equation II \eqref{dfpe2} is obtained from a Markov
process subject to ``white noise'' perturbations. This equation can
also be considered  as  a generalized gradient flow of the free energy on
another metric space $(\mathcal{M},d_{\bar{\bm \Psi}})$ (see Theorem
3 in \cite{CHLZ}). However, the geometry of
$(\mathcal{M},d_{\bar{\bm \Psi}})$ is not smooth. In fact, Chow et al
showed that, in this case, $\mathcal{M}$ is divided into finite
segments, and $d_{\bar{\bm \Psi}}$ is only smooth on each segment.

By the above discussion, we know that the Gibbs distribution ${\bm
\rho}^{*}=(\rho_i^*)_{i=1}^N$ is the unique global equilibrium of
both  equations \eqref{dfpe} and \eqref{dfpe2} in $\mathcal{M}$, and
for any solution ${\bm \rho}(t)$ of both  equations \eqref{dfpe} and
\eqref{dfpe2}, ${\bm \rho}(t)$ will converge to  global equilibrium
${\bm \rho}^{*}$, under the Euclidean metric of $\mathbb{R}^N$,  as
$t\rightarrow \infty$. A natural next question this raises is then
that of the derivation of estimates, such as $O(e^{-ct})$ for a
suitable $c>0$, on the rate of convergence to global equilibrium  for
solutions of  both  equations \eqref{dfpe} and
\eqref{dfpe2}. Answering such a question is one of the main objectives
of the present paper. The rates of convergence towards global equilibrium for
the solution of  these Fokker-Planck equations on a graph are investigated. We
will prove that the convergence is indeed exponential.

In \cite{CHLZ}, the authors introduced several metrics on the space
of probability measures $\mathcal{M}$, including  $d_{\Psi}$,
$d_{\bar{\Psi}}$, an upper bound metric $d_{M}$ and a lower bound
metric $d_{m}$, where the latter two are  independent of the choice
of  potential. These distances are obtained in a sense by
discretizing Felix Otto's calculus -- there is a  certain similarity
between these distances and the 2-Wasserstein distance on the space
of probability measures on $\mathbb{R}^{n}$. For example, the
gradient flow of free energy functional (defined using relative
entropy) in these metric spaces gives rise to the discrete
Fokker-Planck equation in \cite{CHLZ}.
It is worth mentioning that the geodesic of
$d_{\Psi}$ is a discretization of the geodesic equation in
2-Wasserstein distance on the space of probability measures on
$\mathbb{R}^{n}$. For these reasons, sometimes we refer to these as discrete
2-Wasserstein distances.

As an important application of our convergence result, Talagrand-type
inequalities are proved. We will show that  the 2-Wasserstein distance is
bounded from above by the relative entropy: that for all ${\bm \nu}$
absolutely continuous with respect to ${\bm \mu}$, it holds:
\begin{displaymath}
   d_{m}^{2}({\bm \nu},{\bm \mu}) \leq K H({\bm \nu}|{\bm \mu})\,,
\end{displaymath}
where $K$ only depends only on the (reference measure) ${\bm \mu}$.


In recent years, there has been considerable interest in deriving such inequalities in various spaces, with the purpose of studying geometric inequalities connected to concentration of measure and other phenomenon.
On a Riemannian manifold, Otto and Villani \cite{OV} showed (inter alia) that a logarithmic Sobolev inequality implied a Talagrand inequality; this work was soon generalized and simplified by Bobkov, Gentil and Ledoux \cite{BGL}; the latter provided simpler proofs of several previously known results concerning log-Sobolev and transport inequalities. See also \cite{BG} for an earlier work which (along with \cite{OV}) inspired much of the research in this topic.
In subsequent work, Lott and Villani \cite{LV2} used the Hamilton-Jacobi semigroup approach of Bobkov et al \cite{BGL} in showing  that a Talagrand inequality on a measured length space implied a global Poincar\'e inequality, as well as in obtaining (conversely), that spaces satisfying  a certain doubling condition, a local Poincar\'e inequality and a log-Sobolev inequality satisfied a Talagrand inequality.

In discrete spaces, such inequalities are less understood. In part,
the lack of a suitable 2-Wasserstein ($W_2$) distance between
probability measures on a graph has slowed this progress.  M. Sammer
and the last author of this work observed (see \cite{Sammer,ST} for a
proof) that a derivation of Otto-Villani goes through in the context
of a finite graph in yielding the implication that a (weaker) {\em
  modified} log-Sobolev inequality implies a (weaker) Talagrand-type
inequality, relating a $W_1$-distance (rather than a $W_2$) and the
relative entropy.

In the following, we obtain in fact two versions of a Talagrand-type inequality. The
first one is only locally true, which means that the parameter $K$
also depends on the range of ${\bm \nu}$ -- it is
true for all measures ${\bm \nu}$ in a compact neighborhood of ${\bm \mu}$, but
may not be true if ${\bm \nu}$ is arbitrarily far away from ${\bm \mu}$.

The ``global'' Talagrand inequality holds for the ``lower bound''
metric. If the graph $G$ is simple
and connected, then there exists a parameter $K$
that only depends on $\mu$ and certain parameters of $G$. Establishing such an inequality for a suitable notion of a $W_2$-distance on a discrete space continues to be intriguing; particularly, since various people have independently observed that a literal translation of such
an inequality, borrowed from the continuous case, need not hold even
on a 2-point discrete space (see e.g., \cite{M, GRZ}).
However, our results suggest that the metrics introduced in \cite{CHLZ} have a further similarity with the $W_2$-distance on the space of probability measures of the
length space. It is known that on $\mathbb{R}^{n}$, Talagrand
inequality is implied by the log-Sobolev inequality.  It
remains to be seen, however, whether such an implication is true in the present
case.

Independently, a related class of metrics has been
studied by Mielke in \cite{Mie} and Maas in \cite{M}, which are similar to the
Riemannian metrics in \cite{CHLZ} with a constant potential. In the
setting of both \cite{Mie} and \cite{M}, the graphs are assumed to be
associated with an irreducible and reversible Markov kernel. After
essentially finishing this paper, the authors have been informed that
functional inequalities including modified Talagrand inequality and
modified logarithmic Sobolev inequalities associated with the
Riemannian metric studied in \cite{M,  Mie} are independently
investigated in \cite{M2}.


\section{Preliminaries}
In this section, we recall some definitions in graph theory. A {\it graph} is an ordered pair $G =
(V, E)$ where $V=\{a_{1},\cdots,a_{N}\}$ is
the set of vertices and $E$ is the set of
edges. We further assume that the graph $G$ is a simple graph (that is, there are no self loops
or multiple edges) with $|V|\ge 2$, and $G$ is connected. A {\it weighted graph} $(G, w)$
is a pair consisting of a
graph $G=(V,E)$ and a positive real-valued function $w$ of its edges. The function $w$ is
most conveniently described as an $|V|$-by-$|V|$, symmetric, nonnegative matrix
$w= (w_{ij})$ with the property that $w_{ij}>0$ if and only if $(a_i, a_j)\in E$.

Given a graph $G = (V,E)$ with $V=\{ a_1,a_2,\cdots,a_N\}$, we
consider all {\it positive probability distributions} on  $V$:
\[ \mathcal{M} = \left\{ {\bm \rho}=(\rho_{i})^{N}_{i=1}\in \mathbb{R}^N | \sum_{i=1}^N \rho_{i} =
  1 \text{ and } \rho_{i} > 0 \text{ for }i\in \{1,2,\cdots,N\}
\right\}. \]

For ${\bm \mu}=(\mu_{i})^{N}_{i=1}\in \mathcal{M}$ and any map $f: V
\rightarrow \mathbb{R}$, recall the {\it $L^{2}({\bm \mu})$-norm } of $f$ with
respect to $\mu$, denoted by $||f||_{2,{\bm \mu}}$, and given by:
\begin{displaymath}
   ||f||_{2,{\bm \mu}}^{2} := \sum_{i=1}^N(f(a_i))^{2}\mu_{i}\,
\end{displaymath}

Let ${\bm \nu}=(\nu_{i})^{N}_{i=1}\in \mathcal{M}$, then the
{\it relative entropy} $H({\bm \nu}|{\bm \mu})$ of ${\bm \nu}$ with respect to ${\bm \mu}$ is defined by:
\begin{displaymath}
   H({\bm \nu}|{\bm \mu})=\sum_{i=1}^{N}\nu_{i}\log\frac{\nu_{i}}{\mu_{i}},
\end{displaymath}
and we measure the
distance between (the density of) ${\bm \nu}$ and ${\bm \mu}$ using:
\begin{displaymath}
   ||\frac{{\bm \nu}}{{\bm \mu}}-1||^{2}_{2,{\bm \mu}}:= \sum_{i=1}^{N}(\frac{\nu_{i}}{\mu_{i}}-1)^{2}\mu_{i}\,.
\end{displaymath}

Given a graph $G = (V,E)$, its {\it Laplacian matrix} is defined as:
\begin{displaymath}
   \mathcal{L}(G) := D - A\,,
\end{displaymath}
where $D$ is a diagonal matrix with $d_{ii} = \mathrm{deg}(a_{i})$
(number of edges at $a_{i}$), and $A$ is the {\it adjacency matrix} ($A_{ij}
= 1$ if and only if $\{a_{i},a_{j}\} \in E$). As $G$ is a connected simple graph, it is well known that
$\mathcal{L}(G)$ has one $0$ eigenvalue and $|V|-1$ positive eigenvalues.

Given a weighted graph $(G, w)$, its {\it weighted Laplacian
  matrix} is defined as
\begin{displaymath}
   \mathcal{L}(G,w)=diag(\delta_1, \delta_2,\cdots, \delta_{|V|}) - w
\end{displaymath}
with $\delta_i$ denoting the $i$th row sum of $w$. It is well known that
$\mathcal{L}(G, w)$ also has one $0$ eigenvalue and $|V|-1$ positive eigenvalues.

The second smallest eigenvalue $\lambda_2$ of $\mathcal{L}(G)$ (resp. $\mathcal{L}(G,
w)$) is called the {\it spectral gap} of $G$ (resp.$(G,w)$). We remind readers that
there are various standard ways to bound the spectral gap of a
graph. For example for the spectral gap $\lambda_2$ of
$\mathcal{L}(G)$, see \cite{BZ} for the bound
\begin{displaymath}
   \lambda_{2} \geq d_{max} - \sqrt{d_{max}^{2} - d_{min}^{2}}\,,
\end{displaymath}
where $d_{max}$ and $d_{min}$ are the maximum and minimum degrees of
vertices in $G$; similarly  see \cite{LZT}, for

\begin{displaymath}
   \lambda_{2} \geq \frac{2N}{2 + N(N-1)d - 2Md}\,,
\end{displaymath}
where $N$ is the number of vertices, $M$ is the number of edges, and $d$ is the
diameter of $G$; or   \cite{TP} for the bound,

\begin{displaymath}
   \lambda_{2}\geq 2(1-\cos(\frac{\pi}{N}))\,.
\end{displaymath}

\section{The Trend towards Equilibrium}\label{fpe3}

The rate of convergence towards global equilibrium for the solution of  Fokker-Planck equations
\eqref{dfpe} and \eqref{dfpe2} in weighted $L^{2}$ norm are estimated in this
section. We will prove that such convergence is exponential. In
addition, the relative entropy ( with respect to the global
equilibrium ) also has exponential decay.

\subsection{Convergence in weighted $L^{2}$ norm}

The following is our first main result.

\begin{thm}\label{thm-1}  Let $G=(V,E)$ be a graph with its vertex set $V=\{a_1,a_2,\cdots,a_N\}$, edge set $E$, a
given  potential ${\bm \Psi}=(\Psi_i)_{i=1}^N$
  on $V$ and a constant $\beta>0$.
If ${\bm \rho}(t)=(\rho_i(t))_{i=1}^N:[0,\infty)\rightarrow
\mathcal{M}$ is the solution of the Fokker-Planck equation I \eqref{dfpe},
with the initial value ${\bm
\rho}^o=(\rho^o_i)_{i=1}^N\in \mathcal{M}$, then there exists
a constant $C=C({\bm \rho}^o;G,\Psi,\beta)>0$ such that
\begin{equation}
   \label{ggg-eq}
   ||\frac{{\bm \rho}(t)}{{\bm \rho}^{*}} - 1||^{2}_{2,{\bm \rho}^{*}} \le  ||\frac{{\bm \rho}^{o}}{{\bm \rho}^{*}} - 1||^{2}_{2,{\bm \rho}^{*}}\ e^{-Ct}\,,
\end{equation}
 where ${\bm
\rho}^{*}=(\rho_i^*)_{i=1}^N$ is the Gibbs distribution given by
\eqref{eq-1-3}.  In particular,  ${\bm \rho}(t)$ exponentially converges
to global equilibrium: the Gibbs distribution ${\bm \rho}^{*}$ under the
Euclidean metric of $\mathbb{R}^N$  as $t\rightarrow \infty$.
\end{thm}
\begin{proof} Given initial value ${\bm \rho}^o=(\rho_i^o)_{i=1}^N\in
\mathcal{M}$. Let ${\bm
\rho}(t)=(\rho_i(t))_{i=1}^N:[0,\infty)\rightarrow \mathcal{M}$ be
the solution of Fokker-Planck equation I \eqref{dfpe} with initial
value ${\bm \rho}^o\in \mathcal{M}$. For $t\ge 0$, we define
$$L(t)= ||\frac{{\bm \rho}(t)}{{\bm \rho}^{*}} - 1||^{2}_{2,{\bm \rho}^{*}}
= \sum_{i=1}^{N}\frac{(\rho_{i}(t)-\rho_{i}^{*})^{2}}
{\rho_{i}^{*}} ,$$ where ${\bm \rho}^{*}=(\rho_i^*)_{i=1}^N$ is the
Gibbs distribution given by \eqref{eq-1-3}. Now for $t>0$ by
\eqref{dfpe} we have
\begin{align*}
           &\hskip0.5cm \frac{d L(t)}{dt}=\sum_{i=1}^{N}\frac{2(\rho_{i}(t)-\rho_{i}^{*})}{\rho_{i}^{*}}\frac{d\rho_{i}(t)}{dt}\\
                 &=\sum_{i=1}^{N}\frac{2(\rho_{i}(t)-\rho_{i}^{*})}{\rho_{i}^{*}}
                 \Big(\sum_{j\in N(i),\Psi_{j}>\Psi_{i}}\big((\Psi_{j}+\beta\log\rho_{j}(t))-(\Psi_{i}+\beta\log\rho_{i}(t))\big)\rho_{j}(t)\\
                 &+\sum_{j\in N(i),\Psi_{j}<\Psi_{i}} \big((\Psi_{j}+\beta\log\rho_{j}(t))-(\Psi_{i}+\beta\log\rho_{i}(t))\big)\rho_{i}(t)  \\
                   &+\sum_{j\in N(i),
                   \Psi_{j}=\Psi_{i}}\beta(\rho_{j}(t)-\rho_{i}(t))\Big).
\end{align*}
Note that $\Psi_j-\Psi_i=-\beta \log \rho_j^*+\beta \log\rho_i^*$
for $i,j\in \{1,2,\cdots,N\}$ and $\rho_j^*=\rho_i^*$ when
$\Psi_j=\Psi_i$. Combining this with the above equality, we have
\begin{align*}
&\hskip0.5cm \frac{d
L(t)}{dt}\\
&=\sum_{i=1}^{N}\frac{2(\rho_{i}(t)-\rho_{i}^{*})}{\rho_{i}^{*}}
                 \Big(\sum_{j\in N(i),\Psi_{j}>\Psi_{i}}\big((-\beta\log\rho_{j}^{*}+\beta\log\rho_{j}(t))-(-\beta\log\rho_{i}^{*}+\beta\log\rho_{i}(t))\big)\rho_{j}(t)\\
                 &+\sum_{j\in N(i),\Psi_{j}<\Psi_{i}} \big((-\beta\log\rho_{j}^{*}+\beta\log\rho_{j}(t))-(-\beta\log\rho_{i}^{*}+\beta\log\rho_{i}(t))\big)\rho_{i}(t)  \\
                   &+\sum_{j\in N(i),
                   \Psi_{j}=\Psi_{i}}\beta(\frac{\rho_{j}(t)}{\rho_{j}^{*}}-\frac{\rho_{i}(t)}{\rho_{i}^{*}})\frac{\rho_{i}^{*}+\rho_j^*}{2}\Big)\\
                   &=\sum_{i=1}^{N}\frac{2(\rho_{i}(t)-\rho_{i}^{*})}{\rho_i^*}
                 \Big(\sum_{j\in N(i),\Psi_{j}>\Psi_{i}}\beta(\log\frac{\rho_{j}(t)}{\rho_{j}^{*}}-\log\frac{\rho_{i}(t)}{\rho_{i}^{*}})\rho_{j}(t)\\
                 &+\sum_{j\in N(i),\Psi_{j}<\Psi_{i}}\beta(\log\frac{\rho_{j}(t)}{\rho_{j}^{*}}-\log\frac{\rho_{i}(t)}{\rho_{i}^{*}})\rho_{i}(t)+\sum_{j\in N(i),
                   \Psi_{j}=\Psi_{i}}\beta(\frac{\rho_{j}(t)}{\rho_{j}^{*}}-\frac{\rho_{i}(t)}{\rho_{i}^{*}})\frac{\rho_{i}^{*}+\rho_j^*}{2}\Big)
\end{align*}
We denote $ \eta_{i}(t)$ as
$\frac{\rho_{i}(t)-\rho_{i}^{*}}{\rho_{i}^{*}}$ for $t\ge 0$. Then
the above equation can be written as
\begin{align*}
   \frac{dL(t)}{dt}=&\sum_{i=1}^{N}2\eta_{i}(t) \Big(\sum_{j\in N(i),\Psi_{j}>\Psi_{i}}\beta\big(\log(1+\eta_{j}(t))-\log(1+\eta_{i}(t))\big)\rho_{j}(t)\\
                 &+\sum_{j\in N(i),\Psi_{j}<\Psi_{i}}\beta\big(\log(1+\eta_{j}(t))-\log(1+\eta_{i}(t))\big)\rho_{i}(t)  \\
                   &+\sum_{j\in N(i),
                   \Psi_{j}=\Psi_{i}}\beta(\eta_{j}(t)-\eta_{i}(t))\frac{\rho_{i}^{*}+\rho_j^*}{2}\Big).
\end{align*}
For edge $\{ a_i,a_j\}\in E$ with $\Psi_{j}>\Psi_{i}$,
$2\eta_{i}\beta(\log(1+\eta_{j})-\log(1+\eta_{i}))\rho_{j}$ will be
in the above sum at  vertex $a_{i}$;
$2\eta_{j}\beta(\log(1+\eta_{i})-\log(1+\eta_{j}))\rho_{j}$ will be
in the above sum at  vertex $a_{j}$. So we can write the above
equality as
\begin{align}\label{eq-key}
           \frac{dL(t)}{dt}=&-\sum_{\{a_{i},a_{j}\}\in
           E,\Psi_{j}>\Psi_{i}}2\beta(\log(1+\eta_{j}(t))-\log(1+\eta_{i}(t)))(\eta_{j}(t)-\eta_{i}(t))\rho_{j}(t)\nonumber\\
                &-\sum_{\{a_{i},a_{j}\}\in
                E,\Psi_{j}=\Psi_{i}}2\beta(\eta_{j}(t)-\eta_{i}(t))^{2}\frac{\rho_{i}^{*}+\rho_j^*}{2}.
\end{align}
Using \eqref{eq-key} and the following inequality
$$ min
  \{\frac{1}{a},\frac{1}{b}\} \leq\frac{\log a-\log b}{a-b}\leq
max\{ \frac{1}{a}, \frac{1}{b} \}$$ for   $a>0,b>0$ with $a\neq b$,
we have
\begin{align}\label{eq-key1}
           \frac{dL(t)}{dt}\le &-\sum_{\{a_{i},a_{j}\}\in
           E,\Psi_{j}>\Psi_{i}}2\beta(\eta_{j}(t)-\eta_{i}(t))^2 \min \{\frac{1}{1+\eta_i(t)},\frac{1}{1+\eta_j(t)}\}\rho_{j}(t)\nonumber\\
                &-\sum_{\{a_{i},a_{j}\}\in
                E,\Psi_{j}=\Psi_{i}}2\beta(\eta_{j}(t)-\eta_{i}(t))^{2}\frac{\rho_{i}^{*}+\rho_j^*}{2}\nonumber\\
                &=-\sum_{\{a_{i},a_{j}\}\in
           E,\Psi_{j}>\Psi_{i}}2\beta(\eta_{j}(t)-\eta_{i}(t))^2 \min \{\frac{\rho_i^*}{\rho_i(t)},\frac{\rho_j^*}{\rho_j(t)}\}\rho_{j}(t)\nonumber\\
                &-\sum_{\{a_{i},a_{j}\}\in
                E,\Psi_{j}=\Psi_{i}}2\beta(\eta_{j}(t)-\eta_{i}(t))^{2}\frac{\rho_{i}^{*}+\rho_j^*}{2}.
\end{align}
For ${\rm b}=(b_i)_{i=1}^N\in \mathbb{R}^N$, we let
$$m({\rm b})=\min\{
b_i:1\le i\le N\}\text{ and }M({\rm b})=\max\{ b_i:1\le i\le N\}.$$
Put $A(t)= 2\beta \frac{m({\rm \rho}(t))}{M({\rm \rho}(t))}m({\rm
\rho}^*)$ for $t\ge 0$. Then $A(t)>0$ and by \eqref{eq-key1} we have
\begin{align}\label{eq-key2}
           \frac{dL(t)}{dt}\le -A(t)(\sum_{\{a_{i},a_{j}\}\in
           E}(\eta_{j}(t)-\eta_{i}(t))^2).
\end{align}

\noindent
Next we use the following claims (whose proofs appear {\em after} the proof of the present theorem), relating the above right hand side to the spectral gap of the Laplacian matrix $\mathcal{L}(G)$ of graph
$G$.


\begin{claim} \label{claim_one}

\begin{displaymath}
   \sum_{\{a_i,a_j\}\in E} (\eta_{j}(t)-\eta_{i}(t))^{2} \geq
   \frac{\lambda_{2}}{2M({\bm \rho}^{*})} L(t)\,,
\end{displaymath}
where $M({\bm \rho}^{*})$ is the maximal entry of ${\bm \rho}^{*}$
which is less than $1$, and $\lambda_{2}$ is the second smallest
eigenvalue of the Laplacian matrix $\mathcal{L}(G)$ of $G$, or the spectral gap of
$G$.

\end{claim}
We need the following definition before stating the next claim. Let
us denote $$M=\max\{ e^{2|\Psi_i|}:i=1,2,\cdots,N\},$$
$$\epsilon_0=1,$$   and
$$\epsilon_1=\frac{1}{2} \min\left\{ \frac{\epsilon_0}{(1+(2M)^{\frac{1}{\beta}})},
\min\{\rho_i^0:i=1,\cdots,N\}\right\}.$$ For $\ell=2,3,\cdots,N-1$,
we let
$$\epsilon_\ell=\frac{\epsilon_{\ell-1}}{1+(2M)^{\frac{1}{\beta}}}.$$
We define
\begin{align*}
B=\{ &{\bm q}=(q_i)_{i=1}^N\in \mathcal{M}:\sum_{r=1}^\ell q_{i_r}\le 1-\epsilon_{\ell} \text{ where }\ell\in \{1,\cdots,N-1\},\\
&1\le i_1<\cdots <i_\ell \le N \}.
\end{align*}
Then $B$ is a compact subset of $\mathcal{M}$ with respect to  the Euclidean
metric, with
\begin{align*}
\text{int}(B)=\{ &{\bm q}=(q_i)_{i=1}^N\in \mathcal{M}:\sum_{r=1}^\ell q_{i_r}<1-\epsilon_{\ell},\text{ where }\ell\in \{1,\cdots,N-1\},\\
 &1\le i_1<\cdots <i_\ell \le N \}.
\end{align*}
and ${\bm \rho}^0\in \text{int}(B)$.
We have

\medskip
 \begin{claim} \label{claim_two}
 ${\rm \rho}(t)\in B$ for all $t\ge 0$.
 \end{claim}

Using Claim~\ref{claim_one} and \eqref{eq-key2}, we have
\begin{align}\label{eq-key3}
           \frac{dL(t)}{dt}\le -\frac{\lam_2}{2M({\bm \rho}^*)}A(t) \ L(t).
\end{align}
We define
\begin{equation}\label{def-C}
C= \beta \lam_2 \ \frac{m({\bm \rho}^*)}{M({\bm \rho}^*)} \ \frac{1-\epsilon_{L-1}}{\epsilon_1}.
\end{equation}
Clearly  $C>0$ is
dependent on ${\bm \rho}^o$ as well as on $G,\Psi,\beta$, that is $C=C({\bm
\rho}^o;G,\Psi,\beta)$. By the definition of $B$ and Claim~\ref{claim_two}, we
have
\begin{align*}
A(t)&=2\beta \frac{m({\bm \rho}(t))}{M({\bm \rho}(t))}m({\bm
\rho}^*)\ge 2\beta m({\bm \rho}^*) \min \{ \frac{m({\bm q})}{M({\bm
q})}:{\bm q}\in B\} \\
&\ge  2\beta m({\bm
\rho}^*)\frac{1-\epsilon_{L-1}}{\epsilon_1}\\
&= \frac{2M({\bm \rho}^*)}{\lam_2} C
\end{align*}
for $t\ge 0$. Combining this with \eqref{eq-key3}, we get
$\frac{dL(t)}{dt}\le -CL(t)$ for $t>0$. This implies that $L(t)\le
L(0)e^{-Ct}$ for $t\ge 0$. Since $L(0) =  ||\frac{{\bm
    \rho}(0)}{{\bm \rho}^{*}} - 1||^{2}_{2,{\bm \rho}^{*}}$, we have \eqref{ggg-eq}, completing the proof of the theorem, modulo the claim (see below).
\end{proof}
\begin{rem} Given a graph $G=(V,E)$,
a potential ${\bm \Psi}=(\Psi_i)_{i=1}^N$
  on $V$ and a constant $\beta>0$,  the positive constant $C=C({\bm
\rho}^o;G,\Psi,\beta)$ given by \eqref{def-C} is dependent on the initial value ${\bm
\rho}^o\in \mathcal{M}$. In fact  $C({\bm
\rho}^o;G,\Psi,\beta)\rightarrow 0$, when the initial distribution ${\bm
\rho}^o$ converges to the boundary of $\mathcal{M}$.
\end{rem}

We now write a simple observation yielding the first claim used in the proof of the above theorem.
\begin{proof}[Proof of Claim~\ref{claim_one}]
Indeed we have
\begin{align*}
   &\hskip0.5cm \sum_{\{a_i,a_j\}\in E} (\eta_{j}(t)-\eta_{i}(t))^{2}=\frac{1}{2}\sum_{i=1}^N\sum_{j\in N(i)} (\eta_{j}(t)-\eta_{i}(t))^{2}\\
   &\geq \frac{1}{2M({\bm \rho^{*}})}\sum_{i=1}^N\sum_{j\in N(i)} (\eta_{j}(t)-\eta_{i}(t))^{2}\rho^{*}_{i}\geq
   \frac{\lambda_{2}}{2M({\bm \rho^{*}})}Var_{{\bm \rho}^{*}}({\bm \eta}(t))
\end{align*}
the last inequality comes from the Poincare-type inequality (See for example \cite{BT}). In which
\begin{displaymath}
   Var_{{\bm \rho}^{*}}({\bm \eta}(t)) =
   \sum_{i=1}^{N}\rho^{*}_{i}\eta_{i}^{2}(t) - (\sum_{i=1}^{N}\rho^{*}_{i}\eta_{i}(t))^{2}
\end{displaymath}
Note that
\begin{displaymath}
   \sum_{i=1}^{N}\rho^{*}_{i}\eta_{i}(t) = \sum_{i=1}^{N}\rho_{i}(t) -
   \sum_{i=1}^{N}\rho^{*}_{i} = 0
\end{displaymath}
and
\begin{displaymath}
   \sum_{i=1}^{N}\rho_{i}^{*}\eta^2_{i}(t) = L(t)
\end{displaymath}
Hence we have
\begin{displaymath}
   \sum_{\{a_i,a_j\}\in E} (\eta_{j}(t)-\eta_{i}(t))^{2} \geq
   \frac{\lambda_{2}}{2M({\bm \rho}^{*})} L(t)\,.
\end{displaymath}
\end{proof}

Finally, we present the proof of the second claim.

\begin{proof}[Proof of Claim~\ref{claim_two}] We  follow closely the argument in the proof of Theorem 4.1 in \cite{CHLZ}. Since ${\bm \rho}^0\in \text{int}(B)$,
 it is sufficient to show for any ${\bm q}\in B$, the solution ${\bm q}(t)$ through ${\bm q}$ remains in $\text{int}(B)$
for small $t>0$. Let ${\bm q}=(q_i)_{i=1}^N\in B$ and $${\bm
q}(t):[0,c({\bm q}))\rightarrow \mathcal{M}$$ be the  solution to
the equation \eqref{dfpe} with initial value ${\bm q}$ on its
maximal interval of existence. In order to show ${\bm q}(t)\in
\text{int}(B)$ for  small $t>0$, it is sufficient to show that for
any $\ell\in \{1,2,\cdots,N-1\}$ and $1\le i_1<i_2<\cdots i_\ell \le
N$, one has
$$\sum_{r=1}^\ell q_{i_r}(t)< 1-\epsilon_{\ell}\,,$$ for sufficiently small $t>0$.

Given $\ell\in \{1,2,\cdots,N-1\}$ and $1\le i_1<i_2<\cdots i_\ell
\le N$, since ${\bm q}\in B$, we have
$$\sum_{r=1}^\ell q_{i_r}\le 1-\epsilon_{\ell}.$$
Then there are two cases. The first one is  $$\sum_{r=1}^\ell
q_{i_r}<1-\epsilon_{\ell}.$$ It is clear that
$$\sum_{r=1}^\ell q_{i_r}(t)<1-\epsilon_{\ell}\,,$$ for small enough $t>0$ by continuity.

\medskip
The second case is $$\sum_{r=1}^\ell q_{i_r}=1-\epsilon_{\ell}.$$
Let $A=\{i_1,i_2,\cdots,i_\ell\}$ and $A^c=\{1,2,\cdots,N\}\setminus
A$. Then for any $j\in A^c$,
\begin{equation}\label{jle}
q_j\le 1-(\sum_{r=1}^\ell q_{i_r})=\epsilon_\ell.
\end{equation}
Since ${\bm q}\in B$, we have
$$\sum_{j=1}^{\ell-1} q_{s_j}\le 1-\epsilon_{\ell-1},$$
 for any $1\le s_1<s_2<\cdots<s_{\ell-1}\le N$. Hence for each $i\in A$,
\begin{equation}\label{ige}
q_i\ge
1-\epsilon_{\ell}-(1-\epsilon_{\ell-1})=\epsilon_{\ell-1}-\epsilon_{\ell}.
\end{equation}
Combining equations \eqref{jle},\eqref{ige} and the fact  $$\epsilon_\ell\le \frac{\epsilon_{\ell-1}}{1+(2M)^{\frac{1}{\beta}}}\,,$$ one has, for any $i\in A, j\in A^c$,
\begin{equation}\label{eq-lezero}
\Psi_j-\Psi_i+\beta(\log q_j-\log q_i)\le \Psi_j-\Psi_i+\beta(\log
\epsilon_\ell-\log(\epsilon_{\ell-1}-\epsilon_\ell)) \le -\log 2.
\end{equation}
For $\{a_i,a_j\}\in E$, we set
\begin{equation}\label{def-c}
C(\{a_i,a_j\})=\begin{cases} q_j   &\text{ if } \Psi_i<\Psi_j\\
q_i &\text{ if }\Psi_i>\Psi_j\\
\frac{q_{i}-q_{j}}{\log q_{i}-\log q_{j}} &\text{ if }\Psi_i=\Psi_j
\end{cases}.
\end{equation}
Clearly, $C(\{ a_i,a_j\})>0$ for $\{a_i,a_j\}\in E$. Since the graph $G$ is connected, there exists $i_*\in A,j_*\in A^c$
such that $\{a_{i_*},a_{j_*}\}\in E$. Thus
\begin{equation}\label{eq-connect}
\sum \limits_{i\in A,j\in A^c, \{a_i,a_j\}\in E}C(\{a_i,a_j\})\ge C(\{a_{i_*},a_{j_*}\})>0.
\end{equation}
Now by \eqref{eq-lezero} and \eqref{eq-connect}, one has
\begin{eqnarray*}
\frac{\mathrm{d}}{dt}\sum_{r=1}^\ell q_{i_r}(t)\mid_{t=0}&=&\sum_{i\in A}\Big( \sum_{j\in N(i)} C(\{a_i,a_j\})\big(\Psi_j-\Psi_i+\beta(\log q_j-\log q_i)\big)\Big)\\
&=&\sum_{i\in A} \Big( \sum_{j\in A\cap N(i)} C(\{a_i,a_j\})\big(\Psi_j-\Psi_i+\beta(\log q_j-\log q_i)\big)+\\
&& \sum_{j\in A^c\cap N(i)} C(\{a_i,a_j\})\big(\Psi_j-\Psi_i+\beta(\log q_j-\log q_i)\big)\Big)\\
&=&\sum_{i\in A}\Big( \sum_{j\in A^c\cap N(i)} C(\{a_i,a_j\})\big(\Psi_j-\Psi_i+\beta(\log q_j-\log q_i)\big)\Big)\\
&\le& \sum_{i\in A}\big( \sum_{j\in A^c\cap N(i)}  -C(\{a_i,a_j\})\log 2\big)\\
&=&-\log 2 \big( \sum \limits_{i\in A,j\in A^c, \{a_i,a_j\}\in E}C(\{a_i,a_j\})\big)\\
&\le& -C(\{a_{i_*},a_{j_*}\})\log 2 <0.
\end{eqnarray*}
Combining this with the fact $$\sum_{r=1}^\ell
q_{i_r}=1-\epsilon_{\ell},$$ it is clear that
$$\sum_{r=1}^\ell q_{i_r}(t)<1-\epsilon_{\ell}\,,$$ for sufficiently small $t>0$.
This finishes the proof of Claim~\ref{claim_two}.
\end{proof}

Using the same technique, we have the following second main result.

\begin{thm}\label{thm-2}  Let $G=(V,E)$ be a graph with its vertex set $V=\{a_1,a_2,\cdots,a_N\}$, edge set $E$,
a given  potential ${\bm \Psi}=(\Psi_i)_{i=1}^N$
  on $V$ and a constant $\beta>0$.
If ${\bm \rho}(t):[0,\infty)\rightarrow \mathcal{M}$ is the solution
of  Fokker-Planck equation II \eqref{dfpe2},
 with the initial value ${\bm
\rho}^o=(\rho^o_i)_{i=1}^N\in \mathcal{M}$, then
\begin{equation}
\label{ggg-eq-second}
   ||\frac{{\bm \rho}(t)}{{\bm \rho}^{*}} - 1||^{2}_{2,{\bm \rho}^{*}} \le   ||\frac{{\bm \rho}^{o}}{{\bm \rho}^{*}} - 1||^{2}_{2,{\bm \rho}^{*}}\ e^{-Ct}\,,
\end{equation}
where ${\bm \rho}^{*}=(\rho_i^*)_{i=1}^N$ is the Gibbs distribution
given by \eqref{eq-1-3} and $C=\beta \lambda_2 \frac{\min \{ \rho_i^*:1\le i\le N\}}{\max \{ \rho_i^*:1\le i\le N\}}$\,, where $\lambda_2$ is the spectral gap of $G$.
In particular,  ${\bm \rho}(t)$  exponentially
converges to global equilibrium: the Gibbs distribution ${\bm \rho}^{*}$
under the Euclidean metric of $\mathbb{R}^N$,  as $t\rightarrow
\infty$.
\end{thm}
\begin{proof}
Given initial value ${\bm \rho}^0=(\rho_i^0)_{i=1}^N\in
\mathcal{M}$. Let ${\bm
\rho}(t)=(\rho_i(t))_{i=1}^N:[0,\infty)\rightarrow \mathcal{M}$ be
the solution of Fokker-Planck equation II \eqref{dfpe2} with initial
value ${\bm \rho}^0\in \mathcal{M}$. For $t\ge 0$, we define
$$L(t)= ||\frac{{\bm \rho}(t)}{{\bm \rho}^{*}} - 1||^{2}_{2,{\bm \rho}^{*}}
= \sum_{i=1}^{N}\frac{(\rho_{i}(t)-\rho_{i}^{*})^{2}}
{\rho_{i}^{*}} ,$$ where ${\bm \rho}^{*}=(\rho_i^*)_{i=1}^N$ is the
Gibbs distribution given by \eqref{eq-1-3}. Now for $t>0$, by
\eqref{dfpe2}, we have
\begin{align*}
           \frac{dL(t)}{dt}&=\sum_{i=1}^{N}\frac{2(\rho_{i}(t)-\rho_{i}^{*})}{\rho_{i}^{*}}\frac{d\rho_{i}(t)}{dt}\\
                 &=\sum_{i=1}^{N}\frac{2(\rho_{i}(t)-\rho_{i}^{*})}{\rho_{i}^{*}}
                 \Bigl(\sum_{j\in N(i),\bar\Psi_{j}>\bar\Psi_{i}}\big((\Psi_{j}+\beta\log\rho_{j}(t))-(\Psi_{i}+\beta\log\rho_{i}(t))\big)\rho_{j}(t)\\
                 &+\sum_{j\in N(i),\bar\Psi_{j}<\bar\Psi_{i}} \big((\Psi_{j}+\beta\log \rho_{j}(t))-(\Psi_{i}+\beta\log\rho_{i}(t))\big)\rho_{i}(t)\Bigr)  \\
                 &=\sum_{i=1}^{N}\frac{2(\rho_{i}(t)-\rho_{i}^{*})}{\rho_{i}^{*}}
                 \Bigl(\sum_{j\in N(i),\bar\Psi_{j}>\bar\Psi_{i}}\beta(\log\frac{\rho_{j}(t)}{\rho_{j}^{*}}-\log\frac{\rho_{i}(t)}{\rho_{i}^{*}})\rho_{j}\\
                 &+\sum_{j\in
                 N(i),\bar\Psi_{j}<\bar\Psi_{i}}\beta(\log\frac{\rho_{j}(t)}{\rho_{j}^{*}}-\log\frac{\rho_{i}(t)}{\rho_{i}^{*}})\rho_{i}\Bigr)\,,\\
                 \end{align*}
the last equality comes from the fact $\Psi_j-\Psi_i=-\beta \log
\rho_j^*+\beta \log\rho_i^*$ for $i,j\in \{1,2,\cdots,N\}$.

Denoting
$\frac{\rho_{i}(t)-\rho_{i}^{*}}{\rho_{i}^{*}}$  by $ \eta_{i}(t)$, for $t>0$, the
equation will be written as
\begin{align*}
   \frac{dL(t)}{dt}=&\sum_{i=1}^{N}2\beta\eta_{i}(t)\Big(\sum_{j\in N(i),\bar\Psi_{j}>\bar\Psi_{i}}\big(\log(1+\eta_{j}(t))-\log(1+\eta_{i}(t))\big)\rho_{j}(t)\\
                 &+\sum_{j\in
                 N(i),\bar\Psi_{j}<\bar\Psi_{i}}\big(\log(1+\eta_{j}(t))-\log(1+\eta_{i}(t))\big)\rho_{i}(t)\Big)\\
                 &=\sum_{\{a_i,a_j\}\in E,\bar\Psi_j>\bar\Psi_j}
                 2\beta\big(\log(1+\eta_{j}(t))-\log(1+\eta_{i}(t))\big)\big(\eta_{j}(t)-\eta_{i}(t)\big)\rho_{j}(t).
                 \end{align*}
Moreover note that $\eta_{i}(t)=\eta_{j}(t)$ when
$\bar\Psi_{i}=\bar\Psi_{j}$, we have
\begin{align}\label{eq-key-second}
   \frac{dL(t)}{dt}=&-\sum_{\{a_i,a_j\}\in E,\bar\Psi_j>\bar\Psi_i}
                 2\beta\big(\log(1+\eta_{j}(t))-\log(1+\eta_{i}(t))\big)\big(\eta_{j}(t)-\eta_{i}(t)\big)\rho_{j}(t)\nonumber\\
                 &-\sum_{\{a_i,a_j\}\in
                 E,\bar\Psi_j=\bar\Psi_i}2\beta\big(\eta_j(t)-\eta_i(t)\big)^2\frac{\rho_{i}^{*}+\rho_j^*}{2}.
\end{align}
 Using \eqref{eq-key-second} and the following inequality
$$ min
  \{\frac{1}{a},\frac{1}{b}\} \leq\frac{\log a-\log b}{a-b}\leq
max\{ \frac{1}{a}, \frac{1}{b} \}\,,$$ for   $a>0,b>0$ with $a\neq b$,
we have
\begin{align}\label{eq-key1-second}
           \frac{dL(t)}{dt}\le &-\sum_{\{a_{i},a_{j}\}\in
           E,\bar \Psi_{j}>\bar \Psi_{i}}2\beta\big(\eta_{j}(t)-\eta_{i}(t)\big)^2 \min \{\frac{1}{1+\eta_i(t)},\frac{1}{1+\eta_j(t)}\}\rho_{j}(t)\nonumber\\
                &-\sum_{\{a_{i},a_{j}\}\in
                E,\bar \Psi_{j}=\bar \Psi_{i}}2\beta\big(\eta_{j}(t)-\eta_{i}(t)\big)^{2}\frac{\rho_{i}^{*}+\rho_j^*}{2}\nonumber\\
                &=-\sum_{\{a_{i},a_{j}\}\in
           E,\bar \Psi_{j}>\bar \Psi_{i}}2\beta\big(\eta_{j}(t)-\eta_{i}(t)\big)^2 \min \{\frac{\rho_i^*}{\rho_i(t)},\frac{\rho_j^*}{\rho_j(t)}\}\rho_{j}(t)\nonumber\\
                &-\sum_{\{a_{i},a_{j}\}\in
                E,\bar \Psi_{j}=\bar \Psi_{i}}2\beta\big(\eta_{j}(t)-\eta_{i}(t)\big)^{2}\frac{\rho_{i}^{*}+\rho_j^*}{2}.
\end{align}
Now note that $\frac{\rho_{i}^{*}}{\rho_{i}(t)}\ge \frac{\rho_{j}^{*}}{\rho_{j}(t)}$ when $\bar \Psi_{j}>\bar \Psi_{i}$, hence
$\min \{\frac{\rho_i^*}{\rho_i(t)},\frac{\rho_j^*}{\rho_j(t)}\}\rho_{j}(t)=\rho_j^*$ when $\bar \Psi_{j}>\bar \Psi_{i}$.
Combining this with \eqref{eq-key1-second}, we have
\begin{align}\label{eq-key2-second}
           \frac{dL(t)}{dt}\le &-\sum_{\{a_{i},a_{j}\}\in
           E,\bar \Psi_{j}>\bar \Psi_{i}}2\beta\big(\eta_{j}(t)-\eta_{i}(t)\big)^2 \rho_{j}^*\nonumber\\
                &-\sum_{\{a_{i},a_{j}\}\in
                E,\bar \Psi_{j}=\bar \Psi_{i}}2\beta\big(\eta_{j}(t)-\eta_{i}(t)\big)^{2}\frac{\rho_{i}^{*}+\rho_j^*}{2}\nonumber\\
                &\le -2\beta \min\{ \rho_i^*:1\le i\le N\} \sum_{\{a_{i},a_{j}\}\in E} (\eta_{j}(t)-\eta_{i}(t))^2.
\end{align}
By the same argument as in Claim~\ref{claim_one} (used in the proof of Theorem \ref{thm-1}),   we have
$$\sum_{\{a_{i},a_{j}\}\in E} (\eta_{j}(t)-\eta_{i}(t))^2\ge \frac{\lam_2}{2\max \{ \rho_i^*:1\le i\le N\}}L(t).$$
Combining this with \eqref{eq-key2-second},  we get
$\frac{dL(t)}{dt}\le -CL(t)$ for $t>0$, where $C=\beta \lambda_2 \frac{\min\{ \rho_i^*:1\le i\le N\}}{\max \{ \rho_i^*:1\le i\le N\}}$. This implies $L(t)\le
L(0)e^{-Ct}$ for $t\ge 0$, that is \eqref{ggg-eq-second} is true, completing the proof of the theorem.
\end{proof}

\subsection{Exponential decay of the relative entropy}
\label{sec:entropy}

In the following, we show that the entropy decay rate of the Fokker-Planck Equation (II) on $G=(V,E)$ with its vertex set $V=\{a_1,a_2,\cdots,a_N\}$, edge set $E$, a
given  potential ${\bm \Psi}=(\Psi_i)_{i=1}^N$
on $V$ and a constant $\beta>0$ can be bounded in terms of the modified logarithmic Sobolev constant (also known as the ``entropy constant")
$\gamma_0:=\gamma_0(G)$ of the underlying graph $G$:
the optimal $\gamma_0>0$ such that
\begin{equation} \label{mod-log}
   2\gamma_0 \mathrm{Ent}(f) \le \mathcal{E}(f,\log f)\,,
\end{equation}
over all $f:V \to \mathbb{R}$, with $f>0$; recall here the standard notation for the Entropy functional and the Dirichlet form:

\[ \mathrm{Ent} f := \mathrm{Ent}_{{\bm \rho}^{*}}f :=\mathrm{E}_{{\bm \rho}^{*}}(f \log f) - (\mathrm{E}_{{\bm \rho}^{*}} f) \log (\mathrm{E}_{{\bm \rho}^{*}}f)
\,.\]
and
\[ \mathcal{E}(f,\log f) = \sum_{i=1}^N \sum_{j\in N(i)}\big(\log f(a_i) - \log
f(a_j)\big)\big(f(a_i)-f(a_j)\big)\rho^{*}_{i}\\
\,.\]
where ${\bm \rho}^{*}=(\rho_i^*)_{i=1}^N$ is the Gibbs distribution
given by \eqref{eq-1-3}.  See \cite{BT} (where this constant was denoted as $\rho_0$) and references therein, for more information on $\gamma_0$ of a graph and that of a Markov kernel on $G$.

\begin{thm}\label{thm-ent}  Let $G=(V,E)$ be a graph, with its vertex set $V=\{a_1,a_2,\cdots,a_N\}$, edge set $E$, a
given  potential ${\bm \Psi}=(\Psi_i)_{i=1}^N$
  on $V$ and a constant $\beta>0$.
If ${\bm \rho}(t)=(\rho_i(t))_{i=1}^N:[0,\infty)\rightarrow \mathcal{M}$ is the solution
of  Fokker-Planck equation II \eqref{dfpe2}
 with the initial value ${\bm
\rho}^o=(\rho^o_i)_{i=1}^N\in \mathcal{M}$, then
\begin{displaymath}
   H({\bm \rho}(t)|{\bm  \rho}^*)\leq H({\bm \rho}^0|{\bm  \rho}^*)e^{-c t} \text{ for }t\ge 0,
\end{displaymath}
where ${\bm \rho}^{*}=(\rho_i^*)_{i=1}^N$ is the Gibbs distribution
given by \eqref{eq-1-3} and $c=\beta  \gamma_0\frac{\min\{ \rho^{*}_{i}:1\le i\le
N\}}{\max\{ \rho_i^*:1\le i\le N\}}$.
\end{thm}

\begin{proof} Given ${\bm
\rho}^o=(\rho^o_i)_{i=1}^N\in \mathcal{M}$. Let ${\bm \rho}(t)=(\rho_i(t))_{i=1}^N:[0,\infty)\rightarrow \mathcal{M}$
be the solution
of  Fokker-Planck equation II \eqref{dfpe2}
 with the initial value ${\bm \rho}^0$.

Recall that the
relative entropy of ${\bm \rho}=(\rho_i)_{i=1}^N\in \mathcal{M}$ with respect to ${\bm \rho}^{*}$:
\begin{displaymath}
   H({\bm \rho}|{\bm \rho}^{*})=\sum_{i=1}^{N}\rho_{i}\log\frac{\rho_{i}}{\rho^{*}_{i}}\,.
\end{displaymath}
Since the Gibbs distribution is given by
$\displaystyle \rho^{*}_{i} = \frac{1}{K}e^{-\Psi_{i}/\beta}\,,$
we also have
\begin{displaymath}
   \Psi_{i} = - \beta\log\rho^{*}_{i} -\beta \log K\,.
\end{displaymath}
For $t\ge 0$, let $f(t)= {\bm \rho}(t)/{\bm \rho}^{*}$, i.e, $f(t)(a_i)=\rho_i(t)/\rho^*_i$ for $i=1,2,\cdots,N$, we rewrite the relative entropy as usual:
\begin{align*}
\mathrm{Ent} f(t) &:= \mathrm{Ent}_{{\bm \rho}^{*}}f(t)=\mathrm{E}_{{\bm \rho}^{*}}(f \log f) - (\mathrm{E}_{{\bm \rho}^{*}} f) \log (\mathrm{E}_{{\bm \rho}^{*}}f)\\
&= H({\bm \rho}(t)|{\bm \rho}^{*}).
\end{align*}

We write simply $f_i(t)=f(t)(a_i)$.
Then the Fokker-Planck equation II \eqref{dfpe2} becomes
\begin{displaymath}
    \frac{d \rho_{i}(t)} {dt} = \beta\Big(\sum_{\substack{j\in N(i),\\ f_j(t)>f_i(t)}}(\log f_{j}(t)-\log f_{i}(t))\rho_{j}(t)+
    \sum_{\substack{j\in N(i),\\ f_j(t)<f_i(t)}}(\log f_{j}(t)-\log f_{i}(t))\rho_{i}(t)\Big).
\end{displaymath}

Observing that $\frac{a-b}{a}\le \log a-\log b$ when $a>b>0$, we bound the entropy decay by proceeding as follows.
\begin{eqnarray*}
& &\frac{\mathrm{d} \mathrm{Ent}
(f(t))}{\mathrm{d}t}=\frac{\mathrm{d} \big(\sum \limits_{i=1}^N
\rho_i(t)\log \frac{\rho_i(t)}{\rho_i^*}\big)}{\mathrm{d}t}\\
&=&\sum \limits_{i=1}^N \frac{\mathrm{d} \rho_i(t)}{\mathrm{d}t}
\log f_i(t)+\sum \limits_{i=1}^N \frac{\mathrm{d}
\rho_i(t)}{\mathrm{d}t}=\sum \limits_{i=1}^N \frac{\mathrm{d}
\rho_i(t)}{\mathrm{d}t}
\log f_i(t) \\
&=&\beta\sum_{i=1}^{N}\log f_{i}\Big(\sum_{\substack{j\in N(i),\\
f_j(t)>f_i(t)}}(\log f_{j}(t)-\log f_{i}(t))\rho_{j}(t)
+\sum_{\substack{j\in N(i),\\ f_j(t)<f_i(t)}}(\log f_{j}(t)-\log f_{i}(t))\rho_{i}(t)\Big)\\
&=& -\beta\sum_{\substack{\{a_i,a_j\}\in E\\
f_j(t)<f_i(t)}} (\log f_{i}(t) - \log
f_{j}(t))^2 \rho_i(t)\\
&\le & - \beta\sum_{\substack{\{a_i,a_j\}\in E\\
f_j(t)< f_i(t)}}\big(\log f_{i}(t) - \log
f_{j}(t)\big)\frac{(f_{i}(t)-f_{j}(t))}{f_i(t)}\rho_{i}(t)\\
&=& - \beta\sum_{\substack{\{a_i,a_j\}\in E\\
f_j(t)< f_i(t)}}\big(\log f_{i}(t) - \log
f_{j}(t)\big)(f_{i}(t)-f_{j}(t))\rho_i^*\\
&\leq &- \frac{1}{2}\frac{\min\{\rho^{*}_{i}:1\le i\le N\}}{\max\{\rho^{*}_{i}:1\le i\le N\}} \beta\sum_{\substack{\{a_i,a_j\}\in E\\
f_j(t)< f_i(t)}}\big(\log f_{i}(t) - \log
f_{j}(t)\big)(f_{i}(t)-f_{j}(t))(\rho_i^*+\rho_j^*)\\
&=&- \frac{1}{2}\frac{\min\{\rho^{*}_{i}:1\le i\le N\}}{\max\{\rho^{*}_{i}:1\le i\le N\}} \beta\sum_{\{a_i,a_j\}\in E}\big(\log
f_{i}(t) - \log f_{j}(t)\big)(f_{i}(t)-f_{j}(t))(\rho_i^*+\rho_j^*)\\
&=&- \frac{1}{2} \frac{\min\{\rho^{*}_{i}:1\le i\le N\}}{\max\{\rho^{*}_{i}:1\le i\le N\}}\beta  \mathcal{E}(f,\log f)\,,
\end{eqnarray*}
where $\mathcal{E}(\cdot,\cdot)$ is the Dirichlet form of $G=(V,E)$ with respect to
the measure $\rho^{*}$ on $V$.
Now using the definition of the modified log-Sobolev constant (\ref{mod-log}),  we conclude that
\begin{displaymath}
   \frac{\mathrm{d}\mathrm{Ent}(f)}{\mathrm{d}t}\leq - c \mathrm{Ent}(f)\,,
\end{displaymath}
resulting in :
\begin{displaymath}
   \mathrm{Ent}(f(t))\leq \mathrm{Ent}(f(0))e^{-c t}\,,
\end{displaymath}
that is, $H({\bm \rho}(t)|{\bm  \rho}^*)\leq H({\bm \rho}^0|{\bm  \rho}^*)e^{-c t}$
for $t\ge 0$,
where $c=\beta  \gamma_0\frac{\min\{ \rho^{*}_{i}:1\le i\le
N\}}{\max\{ \rho_i^*:1\le i\le N\}}$. This completes the proof of Theorem.
\end{proof}

\section{Talagrand-type Inequalities}

\subsection{Discrete Wasserstein-type metric on $\mathcal{M}$}\label{metric}

Consider a graph $G = (V,E)$ with $V=\{ a_1,a_2,\cdots,a_N\}$. As the
collection of positive probability distributions on $V$, space
$\mathcal{M}$ is defined as in the beginning of Section 2.

\medskip

The tangent space  $T_{{\bm \rho}}\mathcal{M}$ at ${\bm \rho}\in
\mathcal{M}$ has the form
\[ T_{{\bm \rho}}\mathcal{M} = \left\{ {\bm \sigma}=(\sigma_{i})^{N}_{i = 1}\in \mathbb{R}^N | \sum_{i=1}^N
  \sigma_{i} = 0 \right\} \,. \]
It is clear that the standard Euclidean metric on $\mathbb{R}^N$, $d$, is also a Riemannian metric
on $\mathcal{M}$.

Let
\begin{equation*}\label{smooth-phi}
{\bm \Phi}:(\mathcal{M},d)\rightarrow (\mathbb{R}^N,d)
\end{equation*}
be an arbitrary smooth map given by:
$$
{\bm \Phi}({\bm \rho})=(\Phi_i({\bm \rho}))_{i=1}^N, {\qquad \bm
  \rho}\in \mathcal{M} \,.
$$
In the following, we will endow $\mathcal{M}$ with a metric $d_{\bm
\Phi}$, which depends on ${\bm \Phi}$ and the structure of $G$.

\medskip

We consider the function
$$\frac{r_1-r_2}{\log r_1-\log r_2}$$ and  extend it
to the closure of the first quadrant in the plane. Denote
$$e(r_1,r_2)=\begin{cases}
\frac{r_1-r_2}{\log r_1-\log r_2} &\text{ if }r_1\neq r_2 \text{ and }r_1r_2>0\\
0 &\text{ if } r_1r_2=0\\
r_1 &\text{ if } r_1=r_2
\end{cases}.
$$
It is easy to check that $e(r_1,r_2)$ is a continuous function on the
first quadrant and satisfies $\min\{ r_1,r_2\} \le e(r_1,r_2) \le \max\{ r_1,r_2\}$.
For simplicity, we use its original form instead of the
function $e(r_1,r_2)$ in the present paper.


\medskip

The equivalence relation ``$\sim$ " on $\mathbb{R}^N$ is defined as
$${\bm p}\sim{\bm q}\quad  \text{if and only if} \quad
p_1-q_1=p_2-q_2=\cdots=p_N-q_N,$$ and let $\mathcal{W}$ be the
quotient space $\mathbb{R}^N/\sim$. In other words, for ${\bm p}\in \mathbb{R}^N$ we consider its equivalent class
$$[{\bm p}]=\{ (p_1+c,p_2+c,\cdots,p_N+c):c\in \mathbb{R}\}, $$
and all such equivalent classes form the vector space $\mathcal{W}$.

For a given ${\bf \Phi}$, and $[{\bm p}]=[(p_i)_{i=1}^N]\in \mathcal{W}$,
we define an identification $\tau_{\bm \Phi}([{\bm p}])={\bm \sigma}$ from $\mathcal{W}$ to $T_{\bm \rho}\mathcal{M}$ by,
\begin{equation}
   \label{id}
   {\bm \sigma}={\bm p} \mathcal{L}(G, w({\bm \rho})),
\end{equation}
 where
$w({\bm \rho}) = \{w_{ij}({\bm \rho})\}_{i,j = 1}^{N}$
is the weight associated to the original graph $G$:
\begin{displaymath}
   w_{ij}({\bm \rho}) = \left \{ \begin{array}{lll}
\rho_{i}&\mbox{ if } &\Phi_{i}({\bm \rho})>\Phi_{j}({\bm \rho}) , \{a_{i}, a_{j} \} \in E \\
\rho_{j}&\mbox{ if } &\Phi_{i}({\bm \rho})<\Phi_{j}({\bm \rho}) , \{a_{i}, a_{j} \} \in E \\
\frac{\rho_{i} - \rho_{j}}{\log \rho_{i} - \log \rho_{j}}&\mbox{ if }
&\Phi_{i}({\bm \rho})=\Phi_{j}({\bm \rho}) , \{a_{i}, a_{j} \} \in E \\
0 && \mbox{otherwise}
\end{array} \right .
\end{displaymath}
and $\mathcal{L}(G, w({\bm \rho}))$ is the weighted Laplacian matrix of  weighted graph $(G,w({\bm \rho}))$. It is not hard to check that $\tau_{{\bm \Phi}}([ {\bm p}])$ is a
linear isomorphism between $T_{{\bm \rho}} \mathcal{M}$ and
$\mathcal{W}$ (see Lemma 2 in \cite{CHLZ}). Hence ${\bm
\sigma}\in T_{\bm \rho}\mathcal{M}$ can be rewritten as equivalent
classes on $\mathbb{R}^{N}$. For simplicity, we say ${\bm
  \sigma}\simeq [{\bm p}] = [(p_i)_{i=1}^N] $
if $[{\bm p}]:=\tau_{\bm
\Phi}^{-1}({\bm \sigma})\in \mathcal{W} $.

We note that this identification \eqref{id} depends on ${\bm
  \Phi}$, the probability distribution ${\bm \rho}$ and  the structure
of the graph $G$.

\begin{de}\label{de-3.1} By the above identification \eqref{id}, we define an inner product on $T_{\bm \rho}\mathcal{M}$ by:
\begin{eqnarray*}
 g^{\bm \Phi}_{\bm \rho}({\bm\sigma}^{1},{\bm  \sigma}^{2})&=& \sum_{i=1}^{N} p^{1}_{i}\sigma^{2}_{i}=\sum_{i=1}^{N} p^{2}_{i}\sigma^{1}_{i}.
\end{eqnarray*}
\end{de}
It is easy to check that this definition is equivalent to
\begin{equation*}\label{riem-0}
g^{\bm \Phi}_{\bm \rho}({\bm\sigma}^{1},{\bm \sigma}^{2})={\bm p}^1 \mathcal{L}(G, w({\bm \rho})) ({\bm p}^2)^T,
\end{equation*}
for ${\bm \sigma}^{1}=(\sigma_i^1)_{i=1}^N,
{\bm\sigma}^2=(\sigma_i^2)_{i=1}^N \in T_{\rho} \mathcal{M}$, and
$[{\bm p}^1],[{\bm p}^2]\in \mathcal{W}$ satisfying
$${\bm
  \sigma}^{1} \simeq [{\bm p}^1] \text{ and }
{\bm \sigma}^{2} \simeq [{\bm p}^2].$$
In particular,
\begin{equation}\label{riem}
 g^{\bm \Phi}_{\bm \rho}({\bm \sigma},{\bm
   \sigma})={\bm p} \mathcal{L}(G, w({\bm \rho})) {\bm p}^T
\end{equation}
for ${\bm \sigma}\in T_{\rho}\mathcal{M}$, where ${\bm \sigma}\simeq [{\bm p}]$.

The associated distance $d_{{\bm \Phi}}( \cdot, \cdot)$ on
$\mathcal{M}$ is given by
\begin{equation*}\label{geo}
   d_{\bm \Phi}({\bm \rho}^{1},{\bm \rho}^{2}) = \inf_{\gamma} L(\gamma (t))
\end{equation*}
where $\gamma:[0,1]\rightarrow \mathcal{M}$ ranges over all
continuously differentiable curve with $\gamma(0)={\bm \rho}^{1}$,
$\gamma(1)={\bm \rho}^{2}$. The arc length of $\gamma$ is given by
$$L(\gamma(t)) = \int_0^1 \sqrt{ g^{\bm
    \Phi}_{\gamma(t)}(\dot{\gamma}(t),\dot{\gamma}(t))}dt.$$ This
gives the metric space $(\mathcal{M},d_{\bm \Phi})$.

Next, we show that the metric $d_{\bm \Phi}$ is lower bounded.
Given ${\bm \rho}=(\rho_i)_{i=1}^N\in  \mathcal{M}$,
We consider a new identifications
\begin{equation}
   \label{min}
   {\bm \sigma}={\bm p} \mathcal{L}(G, w^m({\bm \rho})),
\end{equation}
 where
$w^m({\bm \rho}) = \{w^m_{ij}({\bm \rho})\}_{i,j = 1}^{N}$
is the weight associated to the original graph $G$:
\begin{displaymath}
   w^m_{ij}({\bm \rho}) = \begin{cases}
\max\{\rho_{i},\rho_j\} &\mbox{ if } \{a_{i}, a_{j} \} \in E \\
0 & \mbox{ otherwise}
\end{cases} .
\end{displaymath}
and $\mathcal{L}(G, w^m({\bm \rho}))$ is the weighted Laplacian matrix of  weighted graph $(G,w^m({\bm \rho}))$.
Similar to the identification \eqref{id}, identifications \eqref{min} is linear isomorphisms between
$T_{\bm \rho}\mathcal{M}$ and $\mathcal{W}$.

Furthermore, they  induce inner product $g_{\bm \rho}^m(\cdot,\cdot)$ on $T_{\bm \rho}\mathcal{M}$.
It is not hard to  see that the map ${\bm \rho}\mapsto g_{\bm \rho}^m$ is smooth. By using the inner products
$g_{\bm \rho}^m$, we can obtain distance $d_m(\cdot,\cdot)$ on $\mathcal{M}$. Then  $(\mathcal{M},d_m)$
 is smooth Riemannian manifold. It is shown in \cite[Lemma
3.4]{CHLZ} that for any smooth map ${\bm \Phi}:(\mathcal{M},d)\rightarrow (\mathbb{R}^N,d)$ and
${\bm \rho}^{1},{\bm \rho}^{2} \in \mathcal{M}$,
 \begin{equation}\label{lower-metric}
   d_{m}({\bm \rho}^{1},{\bm \rho}^{2}) \leq d_{\bm \Phi}({\bm \rho}^{1},{\bm \rho}^{2}) .
\end{equation}

Now we consider two choices
of the function ${\bm \Phi}$  which are related to
Fokker-Planck equation I \eqref{dfpe} and II \eqref{dfpe2} respectively.
Let the potential ${\bm \Psi}=(\Psi_i)_{i=1}^N$ on $V$ be given
and $\beta> 0$, where $\Psi_i$ is the potential on vertex $a_i$.

It then follows from \cite[Section 4]{CHLZ} that Fokker-Planck equation I
\eqref{dfpe} is the gradient flow of free energy $F({\bm \rho})$ on the Riemannian
manifold $(\mathcal{M},d_{\bm \Psi})$, i.e. let
${\bm \Phi}({\bm \rho})\equiv {\bm \Psi}$
where ${\bm \rho}\in \mathcal{M}$. Fokker-Planck equation II
\eqref{dfpe2} is related to inner product $g^{\bar{\bm \Psi}}_{\bm \rho}$, where the new potential
$\bar{\bm \Psi}({\bm \rho})=(\bar{\Psi}_i({\bm \rho}))_{i=1}^N$ is defined by
$$\bar{\Psi}_i({\bm \rho})=\Psi_i+\beta\log \rho_i.$$
Since $g^{\bar{\bm \Psi}}_{\bm \rho}$ is a piecewise smooth function with respect to ${\bm \rho}$, the
space $(\mathcal{M},d_{\bar{\bm \Psi}})$ is a union of finitely many
smooth Riemannian manifolds. Fokker-Planck equation II \eqref{dfpe2}
can also be seen as the generalized gradient flow of $F({\bm \rho})$ on the
metric space $(\mathcal{M},d_{\bar{\bm \Psi}})$ (see \cite[Section 5]{CHLZ}).
By \eqref{lower-metric}, $d_{\bm \Psi}$ and $d_{\bar{\bm \Psi}}$ are lower bounded by $d_m$.

\subsection{Talagrand-type inequalities}

We are now ready to prove  the Talagrand-type inequalities.

\begin{thm}\label{talagrand1}  Let $G=(V,E)$ be a graph with its vertex set $V=\{a_1,a_2,\cdots,a_N\}$ and edge set
$E$. For each ${\bm \mu} = (\mu_{i})_{i=1}^N\in \mathcal{M}$ and  a compact
subset $B$ of $\mathcal{M}$ with respect to the Euclidean metric, there exist a potential function
$\Psi=(\Psi_i)_{i=1}^N$ on $V$ and a constant $K = K(B,{\bm \mu},G)>0$ such that for any ${\bm \nu}=(\nu_{i})_{i=1}^N\in B$, we have  the
following Talagrand-type inequality
\begin{displaymath}
   d^{2}_{\Psi}({\bm \mu},{\bm \nu}) \leq K H({\bm \nu}|{\bm
     \mu}),
\end{displaymath}
where $H({\bm \nu}|{\bm
     \mu})=
     \sum_{i=1}^{N}\nu_{i}\log\frac{\nu_{i}}{\mu_{i}}$.
\end{thm}
\begin{proof} Given ${\bm \mu} = (\mu_{i})_{i=1}^N\in \mathcal{M}$ and a compact
subset $B$ of $\mathcal{M}$ with respect to the Euclidean metric. Let $\Psi_{i} = - \log \mu_{i}$ for $i=1,2,\cdots,N$ and $\beta=1$.
Then $$F({\bm \nu})=H({\bm \nu}|{\bm \mu})\ge 0$$
for any ${\bm \nu}\in \mathcal{M}$.
In the following, we are going to show that there is a constant $K = K(B,{\bm \mu},G)>0$
such that\begin{displaymath}
      d^{2}_{\Psi}({\bm \nu},{\bm \mu}) \leq K H({\bm
     \nu}|{\bm \mu})=K F({\bm \nu})
\end{displaymath}
for any ${\bm \nu }\in B$.

Firstly using \eqref{id} and \eqref{riem}, for ${\bm \sigma}\in T_{\bm{\rho}}\mathcal{M}$ we have
\begin{displaymath}
   ||\bm{\sigma}||^{2} = \bm{p}\mathcal{L}(G,w(\bm{\rho})) \mathcal{L}(G,w(\bm{\rho}))^{T}\bm{p}^{T}
\end{displaymath}
and
 \begin{equation*}
 g^{\bm \Phi}_{\bm \rho}({\bm \sigma},{\bm
   \sigma})={\bm p} \mathcal{L}(G, w({\bm \rho})) {\bm p}^T,
\end{equation*}
where ${\bm \sigma}\simeq [{\bm p}]$ and $||\cdot||$ is the standard Euclidean norm on $\mathbb{R}^N$.

Since $\mathcal{L}(G, w(\bm{\rho}))$ is a real symmetric matrix, we
decompose $\mathcal{L}(G, w(\bm{\rho}))$ into
\begin{displaymath}
   \mathcal{L}(G,w(\bm{\rho})) = Q\Lambda Q^{T}
\end{displaymath}
where $\Lambda$ is a diagonal matrix whose diagonal entries are
eigenvalues of $\mathcal{L}(G,w(\bm{\rho}))$, and $Q$ is a real orthogonal matrix.

Let $\bm{w} =\bm{p}Q$, then we have
\begin{equation}\label{estimate-1}
\begin{split}
    &g^{\Psi}_{\bm \rho}(\bm{\sigma},\bm{\sigma}) = \bm{p}\mathcal{L}(G, w({\bm \rho}))  \bm{p}^{T} =
    \bm{w}\Lambda \bm{w}^{T} \text{ and }\\
   &||\bm{\sigma}||^{2} =  \bm{p}\mathcal{L}(G, w({\bm \rho})) \mathcal{L}(G, w({\bm \rho})) ^{T} \bm{p}^{T} =
   \bm{w}\Lambda^{2}\bm{w}^{T} \,.
\end{split}
\end{equation}
Denote $\lambda_{2}(\bm{\rho})$ and $\lambda_{N}(\bm{\rho})$ the second smallest eigenvalue
and largest eigenvalue of $\mathcal{L}(G, w(\bm{\rho}))$ respectively. Since $\mathcal{L}(G, w({\bm \rho}))$  has one $0$ eigenvalue and $N-1$ positive eigenvalues, it is not hard to see
\begin{displaymath}
    \frac{1}{\lambda_{N}({\bm \rho})}||{\bm \sigma}||^{2} \leq g^{\Psi}_{{\bm \rho}}({\bm \sigma},{\bm \sigma})\leq \frac{1}{\lambda_{2}({\bm \rho})} ||{\bm \sigma}||^{2}\,
\end{displaymath}
by \eqref{estimate-1}.

Let
us denote $$M=\max\{ e^{2|\Psi_i|}:i=1,2,\cdots,N\},$$
$$\epsilon_0=1,$$   and
$$\epsilon_1=\frac{1}{2} \min\left\{ \frac{\epsilon_0}{(1+(2M)^{\frac{1}{\beta}})},
\min_{(\rho_i)_{i=1}^N\in B}\min\{\rho_i: i=1,\cdots,N\}\right\},$$
where $\epsilon_1>0$ as $B$ is compact. For $\ell=2,3,\cdots,N-1$,
we let
$$\epsilon_\ell=\frac{\epsilon_{\ell-1}}{1+(2M)^{\frac{1}{\beta}}}.$$
We define
\begin{align*}
D=\{ &{\bm q}=(q_i)_{i=1}^N\in \mathcal{M}:\sum_{r=1}^\ell q_{i_r}\le 1-\epsilon_{\ell} \text{ where }\ell\in \{1,\cdots,N-1\},\\
&1\le i_1<\cdots <i_\ell \le N \}.
\end{align*}
Then $D$ is a compact subset of $\mathcal{M}$ with respect to  the Euclidean
metric and  with
\begin{align*}
\text{int}(D)=\{ &{\bm q}=(q_i)_{i=1}^N\in \mathcal{M}:\sum_{r=1}^\ell q_{i_r}<1-\epsilon_{\ell},\text{ where }\ell\in \{1,\cdots,N-1\},\\
 &1\le i_1<\cdots <i_\ell \le N \}.
\end{align*}
and $B\subset \text{int}(D)$.

Let
\begin{align*}
   C_{1} = \max_{{\bm \rho}\in D} \{\frac{1}{\lambda_{2}({\bm \rho})}\} \text{ and }
   C_{2} = \min_{{\bm \rho}\in D} \{\frac{1}{\lambda_{N}({\bm \rho})}\},
\end{align*}
Since $\lambda_2,\lambda_N: \mathcal{M}\mapsto (0,+\infty)$ are continuous and $D$ is compact with respect to the Euclidean metric on $\mathcal{M}$,  we have $0< C_{2} \leq C_{1} < +\infty$. It is clear that  $C_{1}$, $C_{2}$ depend only on $B$, ${\bm \mu}$ and $G$, and
\begin{equation}\label{g-bound-1}
   C_2||{\bm \sigma}||^{2} \leq g^{\Psi}_{{\bm \rho}}({\bm \sigma},{\bm \sigma})\leq C_1 ||{\bm \sigma}||^{2}
\end{equation}
for any ${\bm \rho}\in D, {\bm \sigma}\in T_{\bm \rho}\mathcal{M}$.

Now for ${\bm \nu}\in B$, let ${\bm \rho}(t)=(\rho_i(t))_{i=1}^N:[0,+\infty)\rightarrow \mathcal{M}$ is the solution of the Fokker-Planck Equation  \eqref{dfpe} for $\beta = 1$:
\begin{align*}
\begin{aligned}
\label{T1}
 \frac{d \rho_{i}} {dt}=&\sum_{j\in N(i),\Psi_{j}>\Psi_{i}}
 ((\Psi_{j}+\log\rho_{j})-(\Psi_{i}+\log\rho_{i}))\rho_{j}\\
 &+\sum_{j\in N(i),\Psi_{j}<\Psi_{i}} ((\Psi_{j}+\log\rho_{j})-(\Psi_{i}+\log\rho_{i}))\rho_{i}  \\
&+\sum_{j\in N(i), \Psi_{j}=\Psi_{i}}(\rho_{j}-\rho_{i})\,
\end{aligned}
\end{align*}
with initial value ${\bm \nu}$, that is, ${\bm \rho}(0)={\bm \nu}$. Since ${\bm \nu}\in \text{int}(D)$,
we have ${\bm \rho}(t)\in  D$ for all $t\ge 0$, which's proof is similar to the proof of Claim 3.3.

Since the Gibbs distribution given by
\eqref{eq-1-3} is ${\bm \mu}$, Theorem \ref{thm-1} and \eqref{def-C} imply that there exists a constant $C
= C({\bm \mu}, B,G)>0$ such that
\begin{equation*}
\sum_{i=1}^N\frac{(\rho_i(t)-\mu_i)^2}{\mu_i}\le
(\sum_{i=1}^N\frac{(\nu_{i}-\mu_i)^2}{\mu_i})e^{-Ct} \text{
for all }t\ge 0,
\end{equation*}
Moreover let $m = \min \{\mu_i:1\le i\le N\}$ and $M = \max \{\mu_{i}:1\le i\le N\}$, then
\begin{displaymath}
   ||{\bm \rho}(t) - {\bm \mu}||^{2} \leq \frac{M}{m}||{\bm \mu} - {\bm \nu}||^2e^{-Ct}\text{
for all }t\ge 0.
\end{displaymath}
Set $T = \frac{1}{C}\log (\frac{4M}{m})$.
One obtains
\begin{displaymath}
   ||{\bm \rho}(T) - {\bm \mu}||^{2} \leq \frac{1}{4}||{\bm \mu} - {\bm
     \nu}||^{2} \leq \frac{1}{2}(||{\bm \mu} - {\bm \rho}(T)||^{2} +
   ||{\bm \rho}(T) - {\bm \nu}||^{2})\,,
\end{displaymath}
which implies
\begin{displaymath}
   ||{\bm \rho}(T) - {\bm \nu}||^{2} \geq ||{\bm \rho}(T) - {\bm \mu}||^{2}\,.
\end{displaymath}
So after time $T$, the solution of equation \eqref{dfpe} traveled at least half
of the Euclidean distance from ${\bm \nu}$ to ${\bm \mu}$.

Moreover since the Fokker-Planck equation equation \eqref{dfpe}  is the gradient flow of free
energy $F$ under the metric $d_{\Psi}(\cdot,\cdot)$ (see Equation (31) and Theorem 2 in \cite{CHLZ}), we have
\begin{align*}
   \frac{\mathrm{d}F({\bm \rho}(t))}{\mathrm{d}t} =-g^{\Psi}_{{\bm \rho}(t)}(\dot{\bm \rho}(t),\dot{\bm \rho}(t))
\end{align*}
for $t>0$.
By Integrating the previous equality from $0$ to $T$, we have
\begin{eqnarray*}
\begin{aligned}
 F({\bm \nu}) - F({\bm \rho}(T))& = \int_{0}^{T} g^{\Psi}_{{\bm \rho}(t)}(\dot{{\bm \rho}}(t),\dot{{\bm \rho}}(t)) \mathrm{d}t\geq \frac{1}{T}(\int_{0}^{T}\sqrt{ g^{\Psi}_{{\bm \rho}(t)}(\dot{\bm \rho}(t),\dot{\bm \rho}(t))}\mathrm{d}t)^{2}\\
&\geq \frac{1}{T}(\int_{0}^{T}\sqrt{ C_2} ||\dot{\bm \rho}(t)|| \mathrm{d}t)^{2}\geq \frac{C_{2}}{T}||{\bm \nu} - {\bm \rho}(T)||^{2}.
\end{aligned}
\end{eqnarray*}
the last second inequality comes from \eqref{g-bound-1} and the fact that ${\bm \rho}(t)\in D$.
At the same time,
\begin{eqnarray*}
\begin{aligned}
 F({\bm \nu}) - F({\bm \rho}(T))& = \int_{0}^{T} g^{\Psi}_{{\bm \rho}(t)}(\dot{{\bm \rho}}(t),\dot{{\bm \rho}}(t)) \mathrm{d}t\\
&\geq \frac{1}{T}(\int_{0}^{T}\sqrt{ g^{\Psi}_{{\bm \rho}(t)}(\dot{\bm \rho}(t),\dot{\bm \rho}(t))}\mathrm{d}t)^{2}\\
&\geq \frac{1}{T}d_{\Psi}^{2}({\bm \nu},{\bm \rho}(T)).
\end{aligned}
\end{eqnarray*}

Let ${\bm s}(t)=t{\bm \rho}(T)+(1-t){\bm \mu}$ for $t\in [0,1]$. Since $D$ is a convex subset of $\mathbb{R}^N$ and ${\bm \rho}(T),{\bm \mu}\in D$. we have ${\bm s}(t)\in D$ for $t\in [0,1]$. Thus
\begin{align*}
   d^{2}_{\Psi}({\bm \rho}(T),{\bm \mu})
   &\leq (\int_0^1 \sqrt{g^{\Psi}_{{\bm s}(t)}({\bm \rho}(T)-{\bm \mu},{\bm \rho}(T)-{\bm \mu} )}\mathrm{d}t)^2 \\
   &\leq (\int_0^1 \sqrt{C_{1}||{\bm \rho}(T) - {\bm \mu}||^{2}}\mathrm{d}t)^2\\
   &=C_{1}||{\bm \rho}(T) - {\bm \mu}||^{2}
\end{align*}
the last second inequality comes from \eqref{g-bound-1} and the fact that ${\bm s}(t)\in D$.

This gives us the bounds
\begin{eqnarray*}
d^{2}_{\Psi}({\bm \rho}(T),{\bm \mu}) &\leq & C_{1}||{\bm \rho}(T) - {\bm
  \mu}||^{2} \leq C_{1}||{\bm \rho}(T) - {\bm \nu}||^{2}\\
&\leq&\frac{TC_{1}}{C_{2}}(F({\bm \nu}) - F({\bm \rho}(T)))\leq\frac{TC_{1}}{C_{2}}F({\bm \nu})\,,
\end{eqnarray*}
and
\begin{displaymath}
   d_{\Psi}^{2}({\bm \nu},{\bm \rho}(T))\leq T(F({\bm \nu}) -
   F({\bm \rho}(T)))\leq T F({\bm \nu})\,.
\end{displaymath}

In conclusion,
\begin{eqnarray*}
d_{\Psi}^{2}({\bm \mu},{\bm \nu}) &\leq & 2d_{\Psi}^{2}({\bm
  \mu},{\bm \rho}(T)) + 2d_{\Psi}^{2}({\bm \rho}(T),{\bm \nu}) \\
&\leq&(\frac{2TC_{1}}{C_{2}} + 2T)F({\bm \nu}) =
K H({\bm \nu}|{\bm \mu})\,,
\end{eqnarray*}
where $K = (\frac{2TC_{1}}{C_{2}} + 2T)$ is a parameter which only depends on
$B$, $G$ and ${\bm \mu}$.
\end{proof}

The other Talagrand-type inequality is for the ``lower bound'' metric
$d_{m}(\cdot,\cdot)$.

\begin{thm}\label{talagrand2} Let $G=(V,E)$ be a graph with its vertex set $V=\{a_1,a_2,\cdots,a_N\}$ and edge set
$E$. Let $D$ be the maximal degree of $G$ and $\lambda_2$ be the spectral gap of $G$. Given ${\bm \mu} = (\mu_{i})_{i=1}^N\in \mathcal{M}$.  Let $m=\min\{\mu_i:1\le i\le N\}$ and $M=\max\{\mu_i:1\le i\le N\}$. Then for any ${\bm \nu}=(\nu_{i})_{i=1}^N\in \mathcal{M}$, we have the
following Talagrand-type inequality
\begin{displaymath}
   d_{m}^2({\bm \nu},{\bm \mu}) \leq K H({\bm \nu}|{\bm \mu})
\end{displaymath}
where  $K=\frac{M(DN^3+4)}{2\lambda_2 m}\log (\frac{18M}{m^3})$ and $d_{m}(\cdot,\cdot)$ is the lower bounded metric defined in Section \ref{metric}.
\end{thm}
\begin{proof}
Let $\beta=1$, $\Psi_{i} = - \log \mu_{i}$ and $\bar{\Psi}_{i}(\bm
\rho) = -\log \mu_{i} + \log\rho_{i}$ for $i=1,2,\cdots,N$. Then
$F({\bm \rho})=H({\bm
\rho}|{\bm \mu})$ for ${\bm \rho}\in \mathcal{M}$.
In the following, we are going to show that there is a constant $K = K({\bm \mu},G)>0$
such that\begin{displaymath}
      d^{2}_{m}({\bm \nu},{\bm \mu}) \leq K H({\bm
     \nu}|{\bm \mu})=K F({\bm \nu})
\end{displaymath}
for any ${\bm \nu }\in B$.

Now for ${\bm \nu}\in B$, let ${\bm \rho}(t)=(\rho_i(t))_{i=1}^N:[0,+\infty)\rightarrow \mathcal{M}$ is the solution of the Fokker-Planck equation (II) \eqref{dfpe2} with $\beta = 1$,
\begin{align*}
\begin{aligned}
\label{T2}
 \frac{d \rho_{i}} {dt}=&\sum_{j\in N(i),\bar{\Psi}_{j}>\bar{\Psi}_{i}} ((\Psi_{j}+\log\rho_{j})-(\Psi_{i}+\log\rho_{i}))\rho_{j}\\
 &+\sum_{j\in N(i),\bar{\Psi}_{j}<\bar{\Psi}_{i}}
 ((\Psi_{j}+\log\rho_{j})-(\Psi_{i}+\log\rho_{i}))\rho_{i}\,.
\end{aligned}
\end{align*}
with initial value ${\bm \nu}$, that is, ${\bm \rho}(0)={\bm \nu}$.

Let $m = \min\{\mu_{i}:1\le i\le N\}$, $M = \max\{\mu_{i}:1\le i\le N\}$.  Since the Gibbs distribution given by
\eqref{eq-1-3} is ${\bm \mu}$, Theorem \ref{thm-2} implies
\begin{equation*}
\sum_{i=1}^N\frac{(\rho_i(t)-\mu_i)^2}{\mu_i}\le
(\sum_{i=1}^N\frac{(\nu_{i}-\mu_i)^2}{\mu_i})e^{-\lambda_2 \frac{m}{M}t} \text{
for all }t\ge 0,
\end{equation*}
where $\lam_2$ is the spectral gap of $G$. Moreover,
\begin{displaymath}
   ||{\bm \rho}(t) - {\bm \mu}||^{2} \leq \frac{M}{m}||{\bm \mu} - {\bm \nu}||^2e^{-\lambda_2 \frac{m}{M}t}\text{
for all }t\ge 0.
\end{displaymath}
Set $T = \frac{M}{\lambda_2 m}\log (\frac{18M}{m^3})$.
One obtains
\begin{displaymath}
   ||{\bm \rho}(T) - {\bm \mu}||^{2} \leq \frac{m^2}{18}||{\bm \mu} - {\bm
     \nu}||^{2} \leq \frac{m^2}{9}(||{\bm \mu} - {\bm \rho}(T)||^{2} +
   ||{\bm \rho}(T) - {\bm \nu}||^{2})\,,
\end{displaymath}
which implies
\begin{displaymath}
  ||{\bm \rho}(T) - {\bm \mu}||^{2}\leq \frac{1}{8}m^2||{\bm \rho}(T) - {\bm \nu}||^{2}\le \frac{m^2}{4}
\end{displaymath}
as $m\le \frac{1}{N}<1$. Thus ${\bm \rho}(T)\in N({\bm \mu})$, where
$$N({\bm \mu})=\{{\bm \rho}=(\rho_i)_{i=1}^N \in \mathcal{M} : \ \  |\rho_{i} - \mu_{i}|\le
   \frac{m}{2}\,,  \text{ for }i=1,2,\cdots,N\}$$
is a compact convex subset of $\mathcal{M}$ with respect to Euclidean metric.
In other words, after time $T$, the solution of \eqref{dfpe2} travels at least half
of the distance from ${\bm \nu}$ to ${\bm \mu}$ and enters into the
neighborhood $N({\bm \mu})$.
Then we can use the exactly same
method as in Theorem \ref{talagrand1} to estimate $d^{2}_{m}(\cdot,\cdot)$.

Denote $\lambda_{2}(\bm{\rho})$ and $\lambda_{N}(\bm{\rho})$ the second smallest eigenvalue
and largest eigenvalue of $\mathcal{L}(G, w^m(\bm{\rho}))$ respectively. Similar to the proof of Theorem \ref{talagrand1}, we can prove
\begin{equation}\label{estimate-2-m}
    \frac{1}{\lambda_{N}({\bm \rho})}||{\bm \sigma}||^{2} \leq g^{m}_{{\bm \rho}}({\bm \sigma},{\bm \sigma})\leq \frac{1}{\lambda_{2}({\bm \rho})} ||{\bm \sigma}||^{2}\,
\end{equation}
for any ${\bm \sigma}\in T_{\bm \rho}\mathcal{M}$.

For ${\bm \rho}\in \mathcal{M}$, let $\bar{\delta}({\bm \rho})$ be the maximal of the diagonal elements in the
Laplacian matrix $\mathcal{L}(G, w^m(\bm{\rho}))$ and let
\begin{displaymath}
   i_{w^m(\bm{\rho})}(G) = \min_{X\subset V, |X| \leq N/2} (\sum_{i\in X, j\not\in X}w^m_{ij}({\bm \rho})/|X|)\,.
\end{displaymath}
where the minimum is taken over all nonempty
subsets $X$ of $V$ satisfying $|X|\le \frac{N}{2}$. We shall refer to $i_{w^m(\bm{\rho})}(G)$ as {\it the isoperimetric
number} of the weighted graph $(G, w^m(\bm{\rho}))$. Since $G$ is connected and $w^m_{ij}({\bm \rho})\ge \min \{ \rho_i:1\le i\le N\}$ for $\{a_,a_j\}\in E$,
it is not hard to see that
\begin{equation}\label{iso-number}
i_{w^m(\bm{\rho})}(G)\ge \frac{2\min \{ \rho_i:1\le i\le N\}}{N}.
\end{equation}
It follow from
Theorem 2.2 in \cite{BZ} that the spectral gap $\lambda_2({\bm \rho})$ of the weighted graph $(G, w^m(\bm{\rho}))$ satisfies
\begin{displaymath}
   \lambda_{2}({\bm \rho})\geq \bar{\delta}({\bm \rho}) - \sqrt{\bar{\delta}({\bm \rho})^{2} - i_{w^m({\bm \rho})}(G)^{2}}\,,
\end{displaymath}
It then follows from inequality
\begin{displaymath}
   \bar{\delta}({\bm \rho}) - \sqrt{\bar{\delta}({\bm \rho})^{2} - i_{w^m({\bm \rho})}(G)^{2}} \geq \frac{i_{w^m({\bm \rho})}(G)^{2}}{2\bar{\delta}({\bm \rho})}\,
\end{displaymath}
that
\begin{equation}\label{iso-number1}
   \lambda_{2}({\bm \rho})\geq \frac{2(\min \{ \rho_i:1\le i\le N\})^{2}}{DN^{2}}
\end{equation}
by \eqref{iso-number} and the fact that $\bar{\delta}({\bm \rho})\le D$.

Let $C_1=\max \{\frac{1}{\lambda_2({\bm \rho})} : {\bm \rho}\in N({\bm \mu})\}$. Note that $\min \{ \rho_i:1\le i\le N\}\ge  \frac{m}{2}$ for all
${\bm \rho}\in N({\bm \mu})$. We have $C_{1}\le \frac{2DN^2}{m^2}$ by \eqref{iso-number1}.
It is well known that $\lambda_{N}({\bm \rho})\leq N$ for ${\bm \rho}\in \mathcal{M}$ (see for example \cite{OR}).
Let $C_2=\inf \{\frac{1}{\lambda_{N}({\bm \rho})} : {\bm \rho}\in \mathcal{M}\}$. Then $C_2\ge \frac{1}{N}$. Now by \eqref{estimate-2-m}, we have
\begin{equation}\label{dm-lower-b}
   C_2||\bm{\sigma}||^{2}\le g^{m}_{\bm \rho}(\bm{\sigma},\bm{\sigma})
\end{equation}
for all $\bm{\sigma}\in T_{\bm \rho}\mathcal{M}\, , {\bm \rho}\in \mathcal{M}$
and
\begin{equation}\label{dm-upper-b}
    g^{m}_{\bm \rho}(\bm{\sigma},\bm{\sigma})\leq C_1||\bm{\sigma}||^{2}
\end{equation}
for all $\bm{\sigma}\in T_{\bm \rho}\mathcal{M}\, , {\bm \rho}\in N({\bm \mu})$.

Moreover since the Fokker-Planck equation equation \eqref{dfpe2}  is the generalized gradient flow of free
energy $F$ under the metric $d_{\bar{\Psi}}(\cdot,\cdot)$ (see Equation (45) and Theorem 3 in \cite{CHLZ}), we have
\begin{align*}
   \frac{\mathrm{d}F({\bm \rho}(t))}{\mathrm{d}t} =-g^{{\bar \Psi}}_{{\bm \rho}(t)}(\dot{\bm \rho}(t),\dot{\bm \rho}(t))
\end{align*}
for $t>0$.
By integrating the previous equality from $0$ to $T$, we have
\begin{eqnarray*}
\begin{aligned}
 F({\bm \nu}) - F({\bm \rho}(T))& = \int_{0}^{T} g^{{\bar \Psi}}_{{\bm \rho}(t)}(\dot{{\bm \rho}}(t),\dot{{\bm \rho}}(t)) \mathrm{d}t\geq \frac{1}{T}(\int_{0}^{T}\sqrt{ g^{{\bar \Psi}}_{{\bm \rho}(t)}(\dot{\bm \rho}(t),\dot{\bm \rho}(t))}\mathrm{d}t)^{2}\\
&\geq \frac{1}{T}(\int_{0}^{T}\sqrt{ C_2} ||\dot{\bm \rho}(t)|| \mathrm{d}t)^{2}\geq \frac{C_{2}}{T}||{\bm \nu} - {\bm \rho}(T)||^{2}
\end{aligned}
\end{eqnarray*}
the last second inequality comes from \eqref{dm-lower-b}.
At the same time,
\begin{eqnarray*}
\begin{aligned}
 F({\bm \nu}) - F({\bm \rho}(T))& = \int_{0}^{T} g^{\bar \Psi}_{{\bm \rho}(t)}(\dot{{\bm \rho}}(t),\dot{{\bm \rho}}(t)) \mathrm{d}t\\
&\geq \frac{1}{T}(\int_{0}^{T}\sqrt{ g^{\bar \Psi}_{{\bm \rho}(t)}(\dot{\bm \rho}(t),\dot{\bm \rho}(t))}\mathrm{d}t)^{2}\\
&\geq \frac{1}{T}d_{\bar \Psi}^{2}({\bm \nu},{\bm \rho}(T))\geq \frac{1}{T}d_m^{2}({\bm \nu},{\bm \rho}(T))
\end{aligned}
\end{eqnarray*}
the last inequality comes from \eqref{lower-metric}.

Let ${\bm s}(t)=t{\bm \rho}(T)+(1-t){\bm \mu}$ for $t\in [0,1]$. Since $N({\bm \mu})$ is a convex subset of $\mathbb{R}^N$ and ${\bm \rho}(T),{\bm \mu}\in N({\bm \mu})$. we have ${\bm s}(t)\in N({\bm \mu})$ for $t\in [0,1]$. Thus
\begin{align*}
   d^{2}_{m}({\bm \rho}(T),{\bm \mu})
   &\leq (\int_0^1 \sqrt{g^{m}_{{\bm s}(t)}({\bm \rho}(T)-{\bm \mu},{\bm \rho}(T)-{\bm \mu} )}\mathrm{d}t)^2 \\
   &\leq (\int_0^1 \sqrt{C_{1}||{\bm \rho}(T) - {\bm \mu}||^{2}}\mathrm{d}t)^2\\
   &=C_{1}||{\bm \rho}(T) - {\bm \mu}||^{2}
\end{align*}
the last second inequality comes from \eqref{dm-upper-b} and the fact that ${\bm s}(t)\in N({\bm \mu})$.

This gives us the bounds
\begin{eqnarray*}
d^{2}_{m}({\bm \rho}(T),{\bm \mu}) &\leq & C_{1}||{\bm \rho}(T) - {\bm
  \mu}||^{2} \leq \frac{m^2}{8}C_{1}||{\bm \rho}(T) - {\bm \nu}||^{2}\\
&\leq&\frac{m^2TC_{1}}{8C_{2}}(F({\bm \nu}) - F({\bm \rho}(T)))\leq\frac{m^2TC_{1}}{8C_{2}}F({\bm \nu})\,,
\end{eqnarray*}
and
\begin{displaymath}
   d_{m}^{2}({\bm \nu},{\bm \rho}(T))\leq T(F({\bm \nu}) -
   F({\bm \rho}(T)))\leq T F({\bm \nu})\,.
\end{displaymath}

Finally, note that $C_1\le \frac{2DN^2}{m^2}$ and $C_2\ge \frac{1}{N}$. We have
\begin{eqnarray*}
d_{m}^{2}({\bm \nu},{\bm \mu}) &\leq & 2d_{m}^{2}({\bm \nu},{\bm \rho}(T))+2d_{m}^{2}({\bm \rho}(T),{\bm
  \mu}) \\
&\leq&(\frac{m^2TC_{1}}{4C_{2}} + 2T)F({\bm \nu})=T(\frac{m^2C_{1}}{4C_{2}} + 2)F({\bm \nu})\\
&\leq& \frac{M}{\lambda_2 m}\log (\frac{18M}{m^3})\big( \frac{m^2 \frac{2DN^2}{m^2}}{4\frac{1}{N}} + 2\big)F({\bm \nu})\\
&=&
K H({\bm \nu}|{\bm \mu})\,,
\end{eqnarray*}
where $K = \frac{M(DN^3+4)}{2\lambda_2 m}\log (\frac{18M}{m^3})$. This completes the proof.
\end{proof}

\end{document}